\documentclass[12pt]{article}

\usepackage[margin=20truemm]{geometry}

\usepackage{amsmath,amssymb,amsfonts,mathtools,bm,mathrsfs}
\usepackage{physics}
\usepackage{braket}
\usepackage{tikz-cd}
\usepackage{amsthm}

\usepackage{hyperref,cleveref}
    \hypersetup{hidelinks}
\usepackage{bookmark}
\usepackage{booktabs}

\usepackage[maxcitenames=2,style=phys,giveninits=true,doi=false,isbn=false,url=false,eprint=false,alldates=year,useprefix=true,biblabel=brackets,sorting=nyvt]{biblatex}
    \newbibmacro{string+doiurlisbn}[1]{%
      \iffieldundef{doi}{%
        \iffieldundef{url}{%
          \iffieldundef{isbn}{%
            \iffieldundef{issn}{#1}{\href{http://books.google.com/books?vid=ISSN\thefield{issn}}{#1}}}{\href{http://books.google.com/books?vid=ISBN\thefield{isbn}}{#1}}}{\href{\thefield{url}}{#1}}}{\href{http://dx.doi.org/\thefield{doi}}{#1}}}
    \DeclareFieldFormat{title}{\usebibmacro{string+doiurlisbn}{\mkbibemph{#1}}}
    \DeclareFieldFormat[article,incollection]{title}{\usebibmacro{string+doiurlisbn}{\mkbibquote{#1}}}

  \theoremstyle{definition}

    \newtheorem{definition}{Definition}[section]
    \crefname{definition}{Definition}{Definitions}

    \newtheorem{example}[definition]{Example}

    \newtheorem{notation}[definition]{Notation}
    \crefname{notation}{Notation}{Notations}

    \newtheorem{assumption}{Assumption}
    \crefname{assumption}{Assumption}{Assumptions}

  \theoremstyle{plain}

    \newtheorem{theorem}[definition]{Theorem}
    \newtheorem*{theorem*}{Theorem}
    \Crefname{theorem}{Theorem}{Theorems}
    \crefname{theorem}{Theorem}{Theorems}

    \newtheorem{proposition}[definition]{Proposition}
    \Crefname{proposition}{Proposition}{Propositions}
    \crefname{proposition}{Proposition}{Propositions}

    \newtheorem{lemma}[definition]{Lemma}
    \Crefname{lemma}{Lemma}{Lemmas}
    \crefname{lemma}{Lemma}{Lemmas}

    \newtheorem{corollary}[definition]{Corollary}
    \crefname{corollary}{Corollary}{Corollaries}

  \theoremstyle{remark}

    \newtheorem{remark}[definition]{Remark}
    \crefname{remark}{Remark}{Remarks}

  \newcommand{\N}{\mathbb{N}}
  \newcommand{\R}{\mathbb{R}}
  \newcommand{\M}{\MM}
  \newcommand{\X}{\mathscr{X}}
  \newcommand{\XP}{\mathscr{X}_{P_0}^*}
  
  \newcommand{\B}{\mathscr{B}}
  \newcommand{\CC}{C(T)}
  \newcommand{\MM}{M(T)}
  \newcommand{\HP}{H_{P_0}}
  \newcommand{\MMM}{\widetilde{M}}
  \newcommand{\Matp}{\mathrm{Mat}(p,\mathbb{R})}
  \newcommand{\Ck}[2][]{C^{#1}(#2)}
  \newcommand{\CkThM}[1][]{\Ck[#1]{\Theta;\MM}}

  \newcommand{\SN}[1]{\norm{#1}_{\CC}}
  \newcommand{\CkN}[2][]{\norm{#2}_{C^{#1}(\Theta)}}
  \newcommand{\MN}[1]{\norm{#1}_{\text{\rm TV}}}
  \newcommand{\MTN}[1]{\norm{#1}_{\MMM}}
  \newcommand{\OPN}[1]{\norm{#1}_{\mathrm{op}}}
  
  \newcommand{\WkqN}[2][]{\norm{#2}_{W^{#1}(\Theta)}}
  \newcommand{\CkThMN}[2][]{\norm{#2}_{\CkThM[#1]}}
  \DeclarePairedDelimiterX{\innerp}[1]{\langle}{\rangle}{\innpargs{#1}}
  \NewDocumentCommand{\innpargs}{>{\SplitArgument{1}{,}}m}{\innpargsaux#1}
  \NewDocumentCommand{\innpargsaux}{mm}{\ifblank{#1}{
    \ifblank{#2}{{\,\cdot\,}{,}{\,\cdot\,}}
    {{\,\cdot\,}{,}{\mkern2mu#2}}}
    {{#1\mkern2mu}{,}\ifblank{#2}{\,\cdot\,}{\mkern2mu#2}}
  }
  \DeclarePairedDelimiter{\ceil}{\lceil}{\rceil}
  
  \newcommand{\asconv}{\xrightarrow{\mathrm{a.s.}}}
  \newcommand{\pconv}{\xrightarrow{\mathrm{p}}}
  \newcommand{\dconv}{\xrightarrow{\mathrm{d}}}

  \newcommand{\E}{\mathbb{E}}
  
  \newcommand{\K}{\mathcal{K}}
  \newcommand{\idv}{\mathbf{1}}
  \renewcommand{\complement}{{\mathsf{c}}}
  \newcommand{\transp}{\mathsf{T}}
  \newcommand{\Prob}{\mathbb{P}}
  
  \DeclareMathOperator*{\argmax}{arg\,max}
  \DeclareMathOperator*{\argmin}{arg\,min}
  \DeclareMathOperator{\supp}{supp}

  \title{ Maximum likelihood estimation of mean functions for Gaussian processes under small noise asymptotics}
  
  \author{
    Mitsuki Kobayashi\footnote{Research Institute for Science and Engineering, Waseda University} \and 
    Yuto Nishiwaki\footnote{Graduate School of Engineering and Science, Waseda University} \and
    Yasutaka Shimizu\footnote{Department of Applied Mathematics, Waseda University} \and 
    Nobutoki Takaoka\footnotemark[2]
  }
  
  \date{}

\begin{document}

\maketitle

\begin{abstract}
  Maximum likelihood estimators for time-dependent mean functions within Gaussian processes are provided in the context of continuous observations. 
  We find the widest possible class of mean functions for which the likelihood function can be written explicitly. 
  When it is subjected to a small noise asymptotic condition leading to the vanishing of the primary Gaussian noise, we attain local asymptotic normality results, accompanied by insights into the asymptotic efficiency of these estimators. 
  In addition, we introduce M-estimators based on discrete samples, which also leads us to the asymptotic efficiency. 
  Furthermore, we provide quasi-information criteria for model selection analogous to Akaike Information Criteria in discretely observed cases. 

  \begin{description}
    \item[Kerwords:] 
      Gaussian process; mean function; small noise; maximum likelihood estimator; Local asymptotic normality.
    \item[MSC2020:] 60G15; 62F12.
  \end{description}
\end{abstract}

\section{Introduction}

  In recent years, Gaussian Processes (GPs) have emerged as a valuable predictive tool for time series models across various fields, with statisticians primarily focusing on their practical applications. 
  The key advantage of GPs lies in their ability to yield a multidimensional normal distribution as the marginal distribution of sample data. 
  This character simplifies the description of their marginal likelihoods, making many statistical methods based on normality formally applicable. Additionally, GPs align well with machine learning theory, with Bayesian statistical methods receiving particular attention.

  However, it would be important not to overlook a sometimes neglected theoretical aspect: the marginal distribution of GPs depends on both their mean and variance. While considerable attention is paid to their variance in practical applications, less emphasis is placed on estimating their mean function. 
  For instance, a common method in applications involves substituting the mean function with the arithmetic mean of all the data and assuming it to be zero. Although this method may seem practical initially, it lacks theoretical justification. While the law of large numbers might hold if the mean were constant, there's no such guarantee for general GPs, especially if the mean function varies in time. Moreover, \textcite{karvonen2023maximum} argue that maximum likelihood estimators based on marginal likelihoods are ill-posed, particularly in the context of estimating the mean function. Hence, a thorough review of maximum likelihood methods, especially regarding the theory of mean function estimation, is warranted.

  An example of a formulation for estimating the mean function in a continuous-time stochastic process is the drift estimation in stochastic differential equations. 
  The most representative estimation method is maximum likelihood estimation. Under continuous observation, the likelihood function is given parametrically and explicitly in the form of a stochastic integral, which is the Radon-Nikodym derivative of the distribution over the path space of the stochastic process obtained by the Girsanov theorem (see, e.g., \textcite{dermone1995expansion,kutoyants2013statistical}, and references therein). The equivalent likelihood for the GPs case is given by the {\it Cameron-Martin Theorem} in the following form: 
  \begin{equation*}
    \log \dv{P_h}{P_0}\qty(x) = (I^{-1}h)(x) - \frac{1}{2} \norm{ h }_{\HP}^2,
  \end{equation*}
  where $P_h$ is the distribution of GPs with mean function $h$, and $P_0$ is the one with zero mean function, and other detailed notations will be defined in \Cref{subsec:GaussianVectorSetting}, 
  or see, e.g., \textcite[Theorem 5.1]{lifshits2012lectures}. 
  Thus, while the likelihoods are theoretically given in abstract form, it is impossible to write them down as functions of the data in general GPs, as is the case for SDEs, and MLEs cannot be obtained in general.
  Therefore, some authors have been trying to find, through the Malliavin calculus, a class of GPs that yields an explicit representation of the likelihood; see, e.g., \textcite{erraoui2009canonical,nualart1995malliavin,alos2001stochastic}, and \textcite{hult2003approximating} studies a parametric inference for Volterra-type GPs via the Karhunen-Loeve expansion. 
  The first goal of this paper is to obtain an explicit representation of the likelihood in terms of the data by restricting the form of the mean function to a certain class. 
  This results in a broader class of mean functions than the ones obtained above.

  Once an explicit likelihood is obtained, maximum likelihood estimators (MLEs) are easily computed, and the next goal is to find asymptotic properties of MLEs.
  As a matter of fact, to the best of our knowledge, \textcite{tsirelson1983geometrical, tsirelson1986geometrical, tsirelson1987geometrical} are probably the first studies on the asymptotic justification of the maximum likelihood method for GPs. 
  The author adopted the concept of ``small noise asymptotics" which is common in the context of statistics of stochastic differential equations, where the randomness asymptotically vanishes. 
  However, Tsirel'son's estimator is defined in the \emph{space of centered measurable linear functionals}; see  \Cref{subsec:GaussianVectorSetting}, this level of abstraction makes it impractical to discuss the specific form of the likelihood and estimators. 
  We continue in his spirit and realize the construction of estimators more concretely. 

  Now, let us describe briefly the models and the main idea of this paper. The details of the notations are given in the subsequent sections. 

  For a given closed bounded interval $T$, let $X^\varepsilon=(X^\varepsilon_t)_{t \in T}$ be a Gaussian process with the (continuous) mean function $h\in\CC$, 
  and the covariance operator is given by $\varepsilon^2 \K$ with a constant $\varepsilon^2>0$. 
  If the mean function has the form $h=\K\mu_0$, where $\mu_0\in\MM$ that is a finite signed measure on $T$, 
  then the maximum likelihood estimators can be constructed through a measure $\mu_0$ because the likelihood function based on continuous observations is obtained in the form 
  \begin{equation*}
    \log \dv{P_h}{P_0}\qty(x) = \innerp{\mu , X^\varepsilon} -\frac{1}{2} \innerp{ \mu , \K \mu }, 
  \end{equation*}
  where $\innerp{\mu,x} = \int_T x(t)\mu(\dd{t})$. That is, we restrict the class of mean functions to the space $\K(\MM)$ which enables us to obtain the explicit form of the log-likelihood. 
  Thus, since we are assuming that $\K$ is known, our goal is reduced to find an unknown measure $\mu_0$, which gives the unknown mean function $h_0=\K\mu_0$. It would be easy when we have continuous observations denoted by $X^\varepsilon$, and we have 
  \begin{equation*}
    \hat{\mu}_\varepsilon = \argmax_{\mu } \qty{ \innerp{\mu , X^\varepsilon} -\frac{1}{2} \innerp{ \mu , \K \mu } }, \quad
    \hat{h}_\varepsilon \coloneqq \K \mu,   
  \end{equation*}
  where $\argmax$ is taken over in a suitable measure space. However, this ``$\argmax$" requires more careful discussion. This is because the mean function is determined through operators such as $h=\K\mu$, and thus identifiability with respect to $\mu$ becomes an issue, and $\mu$ cannot be uniquely determined in general; see \Cref{thm:functional_convergence} for details. Such a functional maximum likelihood estimator has strong consistency. 

  This methodology is naturally carried over to the parameter estimation as well. That is, by parametrizing $\mu = \mu_\theta$, the MLE for $\theta$ is obtained by 
  \begin{equation*} 
    \hat{\theta}_\varepsilon = \argmax_{\theta\in \Theta} \qty{ \innerp{\mu_\theta,X^\varepsilon}-\frac{1}{2}\innerp{\mu_\theta,\K\mu_\theta}},\quad
    \hat{h}_\varepsilon  \coloneqq \K \mu_{\hat{\theta}_\varepsilon}, 
  \end{equation*} 
  which is shown to be asymptotically efficient as $\varepsilon\to 0$ because the log-likelihood satisfies the LAN conditions under some mild regularities (see \Cref{thm:parametric consistency,thm:parametric asymptotic normality}).
  In addition, since ${\overline{\K(\M)}}^{\SN\cdot}=\supp P_0\subset\CC$, we can say that $\HP$ is bigger enough for statistical inference (see \Cref{prop:InclusionsKXHpSuppP}).

  Also, we can construct a M-estimator of $h_0$ for discrete observation, and the estimator has the asymptotic efficiency (see \Cref{sec:discrete}). By virtue of this asymptotic theory of a GP, we can construct Information Criteria for  mean functions of GPs (see \Cref{subsec:ModelSelection}), which is not formally introduced but has a mathematically rigorous justification in the sense of Theorem \ref{thm:QGAIC}.

  The paper is organized as follows: Section \ref{sec:prelimnaries} covers preliminaries, including notations and definitions of basic mathematical concepts for Gaussian processes; it also contains the definitions necessary to describe the Cameron-Martin theorem and to define the likelihood function as an abstract argument. Section \ref{sec:mainresult} is the main part. Section \ref{sec:functional} gives the MLEs in their functional form in a restricted class of mean functions. The consistency is shown by examining the asymptotic properties under small noise asymptotics. Using this form as a basis, the specific MLE of the parameter is determined when the mean function is parametrized; see Section \ref{sec:parametric}. Under this specific setting, the local asymptotic normality (LAN) of the likelihood function is presented, leading to the asymptotically efficient MLEs. The standard technique of discretizing the continuous-time likelihood leads to the asymptotic normality of the M-estimator, which is again shown to be asymptotically efficient in the sense of minimum asymptotic variance, as well as moment convergence of the estimator.
  In Section \ref{subsec:ModelSelection}, an information criterion based on discrete observations is proposed. A new proof is presented in addition to the standard route for its statistical validity.
  Finally, we discuss some examples that satisfy our regularity conditions and give simulation results that illustrate the asymptotic properties of the estimator.

\section{ Preliminaries}\label{sec:prelimnaries}

  \subsection{Notations and Definitions}

    In this section, we make a list of notations and definitions. 
    \begin{notation}
        Let $T$ be a closed bounded interval, 
        let $C(T)$ be the Banach space of continuous functions on $T$, 
        and let $M(T)$ be its dual space, 
        i.e., the space of finite signed measures on $T$ with the total variation norm $\MN{\cdot}$.
        In addition, we write $\innerp{ \mu, x } \coloneqq\int_T x_t \, \mu(\dd{t})$ for $\mu\in \MM$, $x\in \CC$ or $x\in M(T)^*$.
    \end{notation}

    \begin{notation}[Indicator function]
      For a set $A$, the \emph{indicator function} $\idv_A$ of $A$ is defined as
      \begin{equation*}
        \idv_{A}(x) \coloneqq 
        \begin{cases}
          1 & \text{if $x \in A$}, \\
          0 & \text{if $x \notin A$}.
        \end{cases}
      \end{equation*}
    \end{notation}

    \begin{notation}[Disceretization operator]
        The \emph{discretized measure} $\mu^n$ of $\mu\in\MM$ with respect to a partition $\mathcal{Q}^n$ is defined as 
        \begin{equation*}
            \mu^n \coloneqq \sum_{i=1}^n \mu (T_i) \, \delta_{t_i},
        \end{equation*}
        and the \emph{discretized function} $h^n$ (resp. $K^n$) of $h\in \CC$ (resp. $K \in C(T^2)$) with respect to $\mathcal{Q}^n$ is defined as
        \begin{equation*}
            h^n \coloneqq \sum_{i=1}^n h (t_i) \, \idv_{T_i}, \quad
            K^n \coloneqq \sum_{i,j=1}^n K (t_i,t_j) \, \idv_{T_i\times T_j},
        \end{equation*}
        where $\delta_{t}$ is the Dirac measure centered at $t$, 
        and the partition $\mathcal{Q}^n$  of $T$ is $t_0 < t_1<\dots<t_n$ so that
        \begin{equation*}
          T=\bigcup_{i=1}^n T_i, \quad
          \lim_{n\to\infty} \sup_{i=1,\dots,n} |t_i-t_{i-1}|=0,
          \quad
          T_1\coloneqq[t_0,t_1], \ \  
          T_j\coloneqq(t_{j-1},t_j] \ \  
          (2\le j \le n).
        \end{equation*}
        Note that we regard $(\cdot)^n$ as linear operators from $M(T)$ to itself, $C(T)\to L^\infty(T)$, $C(T^2)\to L^\infty(T^2)$), respectively.
        Similarly, 
        for any integral operator $\mathcal{A}:\MM\to \CC$ of the form
        \begin{equation*}
          \mathcal{A}\mu = \int_T A(s,\cdot) \, \mu(\dd{s})
        \end{equation*}
        with some $A\in C(T^2)$, the \emph{discretized operator}
        $\mathcal{A}^n$ of $\mathcal{A}$ with respect to $\mathcal{Q}^n$ is defined as
        \begin{equation*}
            \mathcal{A}^{n}(\mu) \coloneqq
            \sum_{i,j=1}^n A(t_i, t_j) \, \mu(T_i) \, \idv_{T_j}.
        \end{equation*}
        Note that we regard $(\cdot)^n$ as the linear operator from $B(\MM,\CC)$ to $B(\MM,\MM^*)$ defined by 
        \begin{equation*}
            \innerp{\mu, \mathcal{A}^{n}\nu} = \innerp{\mu^n, \mathcal{A}\nu^n}, \quad 
            \mu,\nu \in \MM,
        \end{equation*}
        where $B(E,F)$ denotes the the space of all bounded linear operators between Banach spaces $E$ and $F$,
        and several properties for the above four discretization operators are summarized in \Cref{apdx:operators}.
    \end{notation}

    \begin{definition}\label{def:DifferentiabilityOfMuTheta} 
      Let $\mathring{\Theta}$ be an open subset of $\R^p$. 
      For a Fr\'echet differentiable function $\mathring{\Theta} \to \MM$, $\theta \mapsto \mu_\theta$,
      we denote the standard Gateaux derivatives, i.e., the directional derivatives along a standard basis vector, by
      \begin{equation*}
        \partial_i\mu_\theta \coloneqq \pdv{\theta_i} \mu_\theta.
      \end{equation*}
      Similarly, denote the second and the third Gateaux derivatives for a three times Fr\'echet differentiable function $\theta\mapsto\mu_\theta$ by
      \begin{equation*}
        \partial_{ij}\mu_\theta \coloneqq \pdv{\theta_i}\pdv{\theta_j} \mu_\theta
        \quad \text{and} \quad 
        \partial_{ijk}\mu_\theta \coloneqq \pdv{\theta_i}\pdv{\theta_j}\pdv{e_k} \mu_\theta.
      \end{equation*}
      In addition, we denote the Fr\'echet derivative of $\mu_\theta$ by $\grad_\theta \mu_{\theta}$ with $\grad_\theta \coloneqq (\partial_1,\dots,\partial_p)^\transp$.
    \end{definition}

    \begin{notation}
      Let $\Theta$ be a closed bounded convex subset of $\mathbb{R}^p$ with smooth boundary, admitting an interior point.
      Then, the Banach space $\CkThM[k]$ is defined by
      \begin{equation*}
        \CkThM[k]
        \coloneqq \Set{ \mu_\bullet:\Theta\to\MM | 
        \parbox{20em}{$\mu_\theta$ is $k$ times Fr\'echet differentiable, and all the $k$-th Gateaux derivatives are uniformly continuous.}
        }
      \end{equation*}
      with norm
      \begin{equation*}
        \CkThMN[k]{ \mu_\bullet }
        \coloneqq \sum_{\abs{\alpha}\leq k} \sup_{\theta\in\Theta} \MN{ \partial_\theta^\alpha \mu_\theta},
      \end{equation*}
      where $\alpha=(\alpha_1,\dots,\alpha_p)\in\mathbb{N}^p$ denotes a multi-index of order $|\alpha|=\alpha_1+\dots+\alpha_p$, and $\partial_\theta^\alpha=\partial_1^{\alpha_1}\cdots\partial_p^{\alpha_p}$ with $\partial_j^{\alpha_j}=\pdv[\alpha_j]{\theta_j}$. 
    \end{notation}
  
  \subsection{Notations for Gaussian Vector}\label{subsec:GaussianVectorSetting}

    In this section, we set up some basic vocabularies and notations with respect to Gaussian vectors. 
    Let $(\Omega,\mathbb{P})$ be a probability space which is enough rich.

    In this section, 
    let $\X$ be a Hausdorff locally convex space, and denote its dual space $\X^*$ with the duality $\innerp{f,X}$ for $f\in \X^*$, $X\in \X$.

    \begin{definition}
      A random vector $X$ taking values in $\X$ is called a \emph{Gaussian vector} on $\X$, 
      if $\langle f,X \rangle$ is a normal random variable for all $f \in \X^*$.
      A vector $h \in \X$ is called the \emph{barycenter} of $X$
      if  $\E[\langle f,X \rangle]=\langle f,h \rangle$ for all $f \in \X^*$,
      and a linear operator $\K:\X^* \to \X$ is called the \emph{covariance operator} of $X$ 
      if $ \mathrm{Cov}\left[\innerp{f,X} , \innerp{g,X}\rangle\right] = \innerp{g, \K f}$ for all $f,g \in \X^*$.
      Particularly, if $h=0$, $X$ is said to be \emph{centered}.
      When $\X=\CC$, we call $h$ the \emph{mean} function of $X$ instead of the barycenter.
    \end{definition}

    \begin{notation}\label{note:NotationOfGaussianVector}
      If a Gaussian vector $X$ has a barycenter $h$ and a covariance operator $\K$, 
      then we write $X \sim N(h,\K)$.
    \end{notation}
    
    \begin{notation}
      $P_0$ denotes the Gaussian measure of $Z$, i.e., $\Prob \circ Z^{-1}$, where $Z \sim N(0,\K)$, and its topological support of $P$ is defined as
      \begin{equation*}
        \supp P\coloneqq \left\{ x \in \X \;\middle|\; P(N_x)>0 \  \text{ for all open neighborhood } N_x \text{ of } x \right\}.
      \end{equation*}
    \end{notation}

    \begin{notation}
        Let $I^*:\X^* \to L^2(\X, P_0)$ be the map $f\mapsto f$,
        and denote by $\XP$ the closure of $I^*( \X^* )$ in $L^2(\X, P_0)$.
        Then, let $I : \XP \to \X $ be the adjoint operator of $I^*$ with
      \begin{equation*}
        \innerp{ g, Iz } = (I^*g, z)_{ \XP } \quad 
        \text{for all} \  g \in \X^*, \  z \in \XP.      
      \end{equation*}
    \end{notation}

    \begin{remark}
        Note that the output $f\in L^2(\X,P_0)$ is defined at $P_0$-a.s. $x\in\X$ and is different from the original $f\in\X^*$,
        and that we regard $\XP$ as the Hilbert space induced by the subspace topology.
        We also note that the well-definedness of $I$ is guaranteed, $\K _0=II^*$, and $\innerp{f,\K g}=(I^*f,I^*g)_{\XP}$ for $f,g\in\X^*$ (see, e.g., Theorem 9.1 in \textcite{lifshits1995gasussian}).
        Some other properties are summarized in \Cref{apdx:operators}.
    \end{remark}

    \begin{definition}
      $h \in \X$ is called \emph{admissible shift} for $P_0$, if the map
      \begin{equation*}
        P_h: \B \to [0,1], \quad A \mapsto P_0(A-h)
      \end{equation*}
      is absolutely continuous with respect to $P_0$.
      \begin{equation*}
        \HP \coloneqq \qty{ h\in\X \;\middle|\; ch\text{ is admissible shift for } P_0, \  \text{for all} \  c\in\R } 
      \end{equation*}
      is called the \emph{kernel} of the Gaussian measure $P_0$ of $X$.
      Then, $\HP$ is regarded as the Hilbert space induced by the linear isomorphism $I$:
      \begin{equation*}
        (h_1, h_2)_{\HP} \coloneqq \qty(I^{-1}h_1, \, I^{-1}h_2)_{\XP}  \quad
        \text{for} \  h_1, h_2\in \HP.
      \end{equation*} 
      (See, e.g., Theorem 3 in \textcite{lifshits1995gasussian}.)
    \end{definition}

    \begin{definition}
        Let $\X=\CC$. 
      The \emph{Covariance function} of $X\in N(0,\K)$ is defined as 
      \begin{equation*}
        K(s,t) \coloneqq \mathrm{Cov}[X_s,X_t], \quad s,t\in T.  
      \end{equation*}
      In this case, $K$ is continuous on $T^2$, and
      \begin{equation*}
          (\K \mu)(s)= \int_T K(s,t) \, \mu(\dd{t}),
          \quad 
          \mu \in \MM.
      \end{equation*}
    \end{definition} 

  \section{M-estimators for mean function of Gaussian process}\label{sec:mainresult}

    This section is devoted to constructing M-estimators and discusses their asymptotic properties. 

    From the Cameron--Martin's theorem (see, e.g., \textcite[Theorem 9.3]{lifshits1995gasussian}) or analogous to \textcite{tsirelson1983geometrical},
    we obtain the exact likelihood form for $N(h,\varepsilon^2 \K)$ with small $\varepsilon>0$:
    \begin{equation}\label{eq:likelihood}
      \log\dv{P^\varepsilon_h}{P^\varepsilon_0}\qty(x)
      =\frac1{\varepsilon^2} \qty{ \innerp{\mu,x} - \frac12 \innerp{\mu,\K\mu} },
    \end{equation}
    where $P^\varepsilon_h$ is probability measure of $N(h,\varepsilon^2 \K)$,  
    $P^\varepsilon_0$ is probability measure of $N(0,\varepsilon^2 \K)$, and $h \in \K(\M)$.
    \Cref{sec:functional} and \Cref{sec:parametric} consider the maximum likelihood estimator using \Cref{eq:likelihood}, while \Cref{sec:discrete} provide the M-estimator for discrete observation using approximation of \Cref{eq:likelihood}.

    All assumptions used in this paper are listed below. Here we use some notations and definitions which have not been presented yet. These notations and definitions are clearly defined later before they are used.

    \begin{assumption}\label{assump:cont}
      The map $\theta\in \Theta \mapsto \mu_\theta \in \MM$ is continuous.
    \end{assumption}
 
    \begin{assumption}\label{assump:isolated}
      $\K \mu_\theta (t) = \K \mu_{\theta_0}(t)$ for all $t\in T$ implies that $\theta \neq \theta_0$.
    \end{assumption}

    \begin{assumption}\label{assump:Theta}
      \(\Theta\subset\R^p\) is compact, and $\theta_0\in \Theta$.
    \end{assumption}

    \begin{assumption}\label{assump:convex_Theta}
      \(\Theta\) is convex and $\theta_0\in \mathring{\Theta}$.
    \end{assumption}

    \begin{assumption}\label{assump:MuIsC^2}
      The map $\mu_\bullet : \Theta \to \MM$ is in $\CkThM[2]$.
    \end{assumption}

    \begin{assumption}\label{assump:Sigma_regular}
      The $p\times p$ matrix $\Sigma$ defined by $\Sigma_{ij}
       \coloneqq \innerp{ \partial_i \mu_{\theta_0}, \K \partial_j\mu_{\theta_0} }$ is regular.
    \end{assumption}

    \begin{assumption}\label{assump:convergent_rate}
      $\varepsilon^{-1}\|K^n -K\|_{L^\infty}\to 0$ as $\varepsilon\to0$, $n\to \infty$.
    \end{assumption}

    \begin{assumption}\label{assump:SmoothBoundaryForMorrey}
      $\Theta$ has a smooth boundary.
    \end{assumption}

    \begin{assumption}\label{assump:MuIsC^3}
      The map $\mu_\bullet : \Theta \to \MM$ is in $\CkThM[3]$.
    \end{assumption}
 
  \subsection{MLE via quotient estimator} \label{sec:functional}

    This section presents ``functional" estimator via maximum likelihood form.  
    In this section, we consider the following situation of small noise asymptotic statistical inference.
    \begin{itemize}
      \item Let $Z\sim \mathcal{N}(0,\K)$ and $K$ be the covariance function of $Z$ corresponding to $\K$.
      \item We observe $X^\varepsilon=h_0+\varepsilon Z \sim N(h_0,\varepsilon^2 \K)$ for some $\varepsilon>0$, where 
      \begin{itemize}
        \item Let $\K$ (resp.\ $K$) be the known covariance operator (resp.\ covariance function).
        \item Let $h_0 \coloneqq h_{\theta_0}=\K \mu_{\theta_0}\in \HP$ be an unknown true mean function with $\mu_0 \coloneqq \mu_{\theta_0}\in\MM$ 
              which are parameterized for $\theta_0\in \Theta$.
      \end{itemize}
    \end{itemize} 
    In this situation, once we can construct $\hat{\mu}$, an estimator of $\mu_0\in \MM$, we can also construct the estimator of $h_0$ by $\hat{h}=\K \hat{\mu}$. 
    Therefore, our first goal is to find $\mu_0$. Since $\mu_0$ to give the mean function $h_0=\K\mu_0$ may not necessarily be unique, we prescribe a quotient space for $\mu_0$. 
    
    \begin{notation}\label{note:quotient}
      Let $\MMM$ be the quotient space
      \begin{equation*}
        \MMM\coloneqq \MM /\ker \K,
      \end{equation*}
      with the quotient norm.
    \end{notation}
    
    In order to define estimators on the quotient space $\widetilde{M}$, the following properties are important.

    \begin{proposition}\label{prop:WelldefContinuousForMLEContrast}
      Let $y\in \overline{\K(\MM)}^{\SN\cdot}$.
      Then, the map
      \begin{equation*}
        [\mu] \mapsto \innerp{\mu,y} - \frac12 \innerp{\mu,\K\mu}
      \end{equation*}
      is well-defined; that is, it does not depend on the choice of a representative $[\mu]$. Moreover, the map is continuous with respect to the quotient norm.
    \end{proposition}

    \begin{proof}
        To see the well-definedness,  
        let $\mu_1$ and $\mu_2$ satisfy $\mu_1-\mu_2\in\ker \K$.
        For fixed $y\in\overline{\K(\M)}$,
        take a sequence $\{\nu_n\}\subset\MM$ such that $\K\nu_n$ converges to $y$ in  $C[0,T]$.
        Since $\K$ is symmetric, 
        it follows that
        \begin{align*}
            \innerp{\mu_1,y} - \innerp{\mu_2,y}
            &= \innerp*{ \mu_1 - \mu_2 , \lim_{n\to\infty} \K \nu_n }
            = \lim_{n\to\infty} \innerp{ \mu_1 - \mu_2 , \K \nu_n } \\
            &= \lim_{n\to\infty} \innerp{ \nu_n , \K (\mu_1 - \mu_2) }
            = 0,
        \end{align*}
        and that
        \begin{equation*}
            \innerp{ \mu_1,\K\mu_1 } - \innerp{ \mu_2,\K\mu_2 }
            = \innerp{ \mu_1+\mu_2,\K(\mu_1-\mu_2) } = 0, 
        \end{equation*}
        which shows the well-definedness.
    
        For the continuity, we suppose that $[\mu_n]\to[\mu]$ in $\MMM$ as $n\to \infty$. 
        Then,
        \begin{align*}
            &\abs{ \qty( \innerp{\mu_n,y} - \frac12 \innerp{\mu_n,\K\mu_n} ) 
                - \qty( \innerp{\mu,y} - \frac12 \innerp{\mu,\K\mu} ) }\\
            &\qquad\qquad\qquad\qquad = \abs{ \innerp*{ \mu_n-\mu, y } } 
                + \frac12 \abs{ \innerp*{ \mu_n+\mu, \K (\mu_n-\mu) } }
        \end{align*}
        When we use $\{\nu_m\}$ taken above again, it follows from \Cref{lem:sym_schwartz} that
        \begin{align*}
            \abs{ \innerp*{ \mu_n-\mu, y } }
            &= \abs{ \innerp*{ \mu_n-\mu, \lim_{m\to\infty} \K \nu_m } } 
            = \lim_{m\to\infty} \abs{ \innerp*{ \mu_n-\mu, \K \nu_m } } \\
            &\leq \lim_{m\to \infty} \SN{\K \nu_m} \MTN{[\mu_n-\mu]}
            = \SN{y} \MTN{[\mu_n]-[\mu]}\to 0,\quad n\to \infty.
        \end{align*}
        Moreover, it follows that 
        \begin{equation*}
            \abs{ \innerp*{ \mu_n+\mu, \K (\mu_n-\mu) } }
            \leq \OPN{\K} \MTN{[\mu_n+\mu]}\MTN{[\mu_n-\mu]}\to 0,
        \end{equation*}
        as $n\to \infty$. This ends the proof.
        \end{proof}
    
    According to \Cref{prop:WelldefContinuousForMLEContrast}, we make a contrast function $\widetilde{\Phi}_\varepsilon$ to estimate $\mu_0$ as follows: 
    \begin{equation*}
      \widetilde{\Phi}_{\varepsilon}([\mu]) \coloneqq \innerp{ \mu,X^\varepsilon } - \frac12 \innerp{ \mu,\K\mu }.
    \end{equation*} 
    Note that
    \[
      \widetilde{\Phi}_{\varepsilon}([\mu]) \asconv \widetilde{\Phi}([\mu]) \coloneqq \innerp{\mu,h_0} - \frac12\innerp{ \mu,\K\mu }
      \quad (\varepsilon\to0).
    \]
    
    \begin{proposition} \label{prop:uniquemaximum}
      The map $[\mu] \mapsto \innerp{\mu,h_0} - \frac12 \innerp{\mu,\K\mu}$ admits the unique maximum point $[\mu_0]$.
    \end{proposition}
    
    \begin{proof}
      Recall that $h_0=\K\mu_0$ set earlier in this section.
      Since 
      \begin{equation*}
        \innerp{\mu,\K\mu_0} - \frac12 \innerp{\mu,\K\mu} 
        = - \frac12 \innerp{\mu-\mu_0,\K(\mu-\mu_0)}
          + \frac12 \innerp{\mu_0,\K\mu_0},
      \end{equation*}
      $[\mu]\in\MMM$ achieves the maximum of $[\mu]\mapsto\innerp{\mu,y} - \frac12 \innerp{\mu,\K\mu}$ if and only if $\innerp{\mu-\mu_0,\K(\mu-\mu_0)}=0$.
      Thus, when $[\mu]\in\MMM$ is a maximum point, its representative $\mu$ always satisfies
      \begin{equation*}
        \norm{I^*(\mu-\mu_0)}_{\HP}^2
        = \innerp{\mu-\mu_0,\K(\mu-\mu_0)}
        = 0.
      \end{equation*}
      This implies that $\K(\mu-\mu_0)=II^*(\mu-\mu_0)=0$, i.e., $[\mu]=[\mu_0]$.
    \end{proof}

          \begin{lemma}\label{prop:InclusionsKXHpSuppP}
            $\overline{\K(\MM)}^{\SN\cdot}
              = \overline{H_{P_0}}^{\SN\cdot} 
              = \supp P_0$.
      \end{lemma}
    
      \begin{proof}
          Since $\K(\MM)\subset \overline{H_{P_0}}^{\SN\cdot} = \supp P_0$ (cf. Theorem 9.6 in \textcite{lifshits1995gasussian}),
          it is sufficient to show that $\HP\subset\overline{\K(\M)}^{\SN\cdot}$.
        Take an arbitrary $h\in\HP$.
        Then, there exist $z\in\XP$ and $\mu_n\in\MM$ such that $h=Iz$ and $\norm{z-I^*\mu_n}_{\XP}\to0$.
        Since $\K=II^*$, it follows from \Cref{prop:OpnormOfIandI*} that
        \begin{equation*}
          \SN{h-\K \mu_0} 
          \leq \norm{\K}_{\mathrm{op}}^{1/2} \norm{z-I^*\mu_n}_{\XP} 
          \to 0,
        \end{equation*}
        as $n\to\infty$.
      \end{proof}

    \begin{proposition} \label{prop:nonempty1}
      Assume
      $\widetilde{\Theta} \subset \MMM$ is compact and $[\mu_0] \in \widetilde{\Theta}$,
      then $\displaystyle \argmax_{[\mu] \in \widetilde{\Theta}} \widetilde{\Phi}_\varepsilon(\theta)$ is nonempty.
    \end{proposition}
  
    \begin{proof}
      Since $\widetilde{\Theta}$ is closed and bounded, it follows from  \Cref{prop:WelldefContinuousForMLEContrast,prop:InclusionsKXHpSuppP} and the maximum principle that $\widetilde{\Phi}_\varepsilon$ admits a maximum point on $\widetilde{\Theta}$ almost surely. 
    \end{proof}
  
    By virtue of \Cref{prop:nonempty1}, we define an MLE $\hat{h}_\varepsilon$ of the mean function via the \emph{quotient estimator} $[\hat{\mu}_\varepsilon]$ as
    \begin{equation*}
        \hat{h}_\varepsilon \coloneqq \K \hat{\mu}_\varepsilon
        \quad \text{via} \quad
      [\hat{\mu}_\varepsilon] \coloneqq \argmax_{[\mu]\in \widetilde{\Theta}} \widetilde{\Phi}_\varepsilon.
    \end{equation*} 
  
    \begin{theorem}[Consistency]\label{thm:functional_convergence}
      Under the same assumptions as in \Cref{prop:nonempty1},
      \begin{equation*}
          [\hat{\mu}_\varepsilon] \asconv [\mu_0] \ \  \text{in} \ \MMM
          \quad
          \hat{h}_\varepsilon \asconv h_0 \ \  \text{in} \ \CC \  \text{and} \  \HP.
      \end{equation*}
      as $\varepsilon\to0$.
    \end{theorem}
    
    \begin{proof}
      Since $\widetilde{\Theta}$ is bounded, 
      \begin{equation}\label{eq:func_conti_contrast_uni_conv}
        \sup_{[\mu] \in \widetilde{\Theta}} \abs{ \widetilde{\Phi}_{\varepsilon} ([\mu])- \widetilde{\Phi}([\mu]) }
        = \sup_{[\mu] \in \widetilde{\Theta}} \varepsilon \abs{ \innerp*{ \mu, Z } }
        \asconv 0
      \end{equation} 
      as $\varepsilon \to 0$. 
      It follows from \Cref{prop:WelldefContinuousForMLEContrast,prop:uniquemaximum} that,
      for any $\delta>0$,
      \begin{equation*}
        \widetilde{\Phi}\left([\mu_0]\right) > \sup_{\MTN{[\mu]-[\mu_0]}\geq\delta} \widetilde{\Phi}\left([\mu]\right).
      \end{equation*}
      By the definition of $\hat{\theta}_{\varepsilon}$, 
      \begin{equation*}
        \widetilde{\Phi}_{\varepsilon}([\hat{\mu}_{\varepsilon}]) \ge \widetilde{\Phi}_{\varepsilon}([\mu_0]),
      \end{equation*}
      Thus, it follows from \Cref{prop:M-estimator} that $[\hat{\mu}_{\varepsilon}] \asconv [\mu_0]$ in $\MMM$
      as $\varepsilon\to0$.
      By \Cref{lem:sym_schwartz}, the map $[\mu]  \mapsto \K\mu$ is continuous as the map $\MMM \to \CC$ or $\MMM \to \HP$,
      and therefore, it follows by continuous mapping theorem that
      $\hat{h}_\varepsilon \asconv h_0$ in $\CC$
      as $\varepsilon \to 0$.
    \end{proof}

  \subsection{Parametric inference under continuous observation}\label{sec:parametric} 

    This section considers parametric inference from continuous observation by applying \Cref{sec:functional}.

    \subsubsection{Construction of estimators and Consistency} \label{subsec:conti_consistency}

      In this section and subsequent sections, we consider the case that the true mean function is parameterized as $h_0=\K \mu_{\theta_0} \in \K(\M)$ ($\theta_0 \in \Theta \subset \R^p$: parameter space).
      We construct the maximum likelihood estimator of $h_{\theta_0}$ through estimating $\theta_0$ and $\mu_{\theta_0}$.

      \begin{definition}  
        Let $\Phi_\varepsilon$ and $\Phi$ be the following functions: 
        \begin{align*}
            \Phi_\varepsilon(\theta) 
            &\coloneqq \varepsilon^2 \qty{ \log \dv{P^\varepsilon_{h_\theta}}{P^\varepsilon_0}\qty(X^\varepsilon) } 
                = \innerp*{\mu_\theta,X^\varepsilon} - \frac12 \innerp{ \mu_\theta,\K\mu_\theta }, \\
            \Phi(\theta) 
            &\coloneqq \innerp{ \mu_\theta,h_0 } - \frac12 \innerp{ \mu_\theta,\K\mu_\theta }
            = -\frac12 \innerp{ \mu_\theta - \mu_{\theta_0}, \K (\mu_\theta - \mu_{\theta_0})} + \frac12 \innerp{ \mu_{\theta_0}, \K \mu_{\theta_0} }.
        \end{align*}
        Note that $\Phi$ is the almost sure limit of $\Phi_\varepsilon$ as $\varepsilon\to 0$. 
      \end{definition}

      Under \Cref{assump:Theta,assump:cont,assump:isolated}, we have the following two propositions by the similar arguments as in \Cref{prop:WelldefContinuousForMLEContrast,prop:uniquemaximum,prop:nonempty1}.
    
      \begin{proposition}\label{prop:ContinuousContrast_UniqueMaximum}
        Under \Cref{assump:cont}, for $y\in\overline{\K(\M)}^{\SN\cdot}$, the map 
        \begin{equation*}
            \Theta \to \mathbb{R}, \quad 
            \theta\mapsto\innerp{\mu_\theta,y} - \frac12 \innerp{\mu_\theta,\K\mu_\theta}
        \end{equation*}
        is continuous.
        If \Cref{assump:isolated} also holds, then $\Phi$ 
        admits the unique maximum point $\theta_0$.
      \end{proposition}
  
      \begin{proposition}\label{prop:AbleToChooseEstimator}
        Under \Cref{assump:isolated}, $\argmax_{\theta\in\Theta} \Phi_\varepsilon(\theta)$ is nonempty.
      \end{proposition}
          
      By virtue of \Cref{prop:AbleToChooseEstimator}, we define an MLE $\hat{h}_\varepsilon$
      \begin{equation*}
          \hat{\theta}_\varepsilon \coloneqq \argmax_{\theta\in \Theta} \Phi_\varepsilon (\theta),
          \quad 
          \hat{\mu}_\varepsilon\coloneqq \mu_{\hat{\theta}_\varepsilon}
          \quad \text{and} \quad
          \hat{h}_\varepsilon \coloneqq \K \hat{\mu}_\varepsilon.  
      \end{equation*} 
    
      \begin{theorem}[Consistency]\label{thm:parametric consistency}     
        Under \Cref{assump:Theta,assump:cont,assump:isolated},
        it holds that
        \begin{equation*}
            \hat{\theta}_\varepsilon \asconv \theta_0 \ \  \text{in} \  \mathbb{R}^p,
            \quad 
            \hat{\mu}_\varepsilon \asconv \mu_0 \ \  \text{in} \  \MM,
            \quad 
            \hat{h}_\varepsilon \asconv h_0 \ \  \text{in} \  \CC \  \text{and} \  \HP.
        \end{equation*}
        as $\varepsilon\to 0 $.
      \end{theorem}   
    
      \begin{proof}
        By the same argument as in the proof of \Cref{thm:functional_convergence}, 
        it follows from Propositions \ref{prop:M-estimator} and \ref{prop:ContinuousContrast_UniqueMaximum} that
        $\hat{\theta}_{\varepsilon} \asconv \theta_0$ as $\varepsilon \to 0$.
        Then, the continuous mapping theorem with \Cref{assump:cont} yields the consequence.
      \end{proof}

    \subsubsection{Asymptotic Normality}\label{subsec:conti_AN}

      We will show the asymptotic normality of $\hat{\theta}_\varepsilon$ under some additional regularity conditions 
      \Cref{assump:convex_Theta,assump:MuIsC^2,assump:Sigma_regular}.
      
      \begin{remark}\label{note:SigmaIsGramMatrix}
        Since $\K = I I^*$, the $(i,j)$ entry of $\Sigma$ defined in \Cref{assump:Sigma_regular} can be expressed as 
        \begin{equation*}
          \Sigma_{ij} = \innerp{ \partial_i \mu_{\theta_0}, I I^* (\partial_j\mu_{\theta_0}) } 
          = \qty( I^*(\partial_j\mu_{\theta_0}), I^*(\partial_j\mu_{\theta_0}) )_{\XP},
        \end{equation*}
        which is commonly refered to as a Gram matrix. If and only if $(I^*(\partial_i\mu_{\theta_0}))_{1\le i\le p}$ is linearly independent in $\XP$, then \Cref{assump:Sigma_regular} is satisfied (see \Cref{ex:density} for some concrete examples). 
      \end{remark}
      
      \begin{theorem}[Asymptotic normality]
        \label{thm:parametric asymptotic normality} 
        Under \Cref{assump:Theta,assump:cont,assump:isolated,assump:convex_Theta,assump:MuIsC^2,assump:Sigma_regular},
        it holds that
        \begin{equation*}
          \varepsilon^{-1} ( \hat{\theta}_\varepsilon - \theta_0 ) 
          \pconv \Sigma^{-1} \innerp*{ \grad_\theta \mu_{\theta_0}, Z }
          \sim N(0,\Sigma^{-1})
          \quad \text{as} \ 
          \varepsilon\to 0.
        \end{equation*}
      \end{theorem}

      Before going to the proof of \Cref{thm:parametric asymptotic normality},
      we refer to some notations and lemmas.
      
      \begin{notation}
        Under \Cref{lem:measure_partial_change} and \Cref{assump:MuIsC^2}, $\Phi_\varepsilon:\Theta\to \R$ is twice differentiable, and thus we put
        \begin{equation*}
          J_\varepsilon 
          \coloneqq \int^1_0 \nabla^2_\theta \Phi_\varepsilon \qty( (1-u) \theta_0 + u\hat{\theta}_\varepsilon ) \dd{u}
          \quad \text{and} \quad
          A_\varepsilon \coloneqq A'_\varepsilon \cap A''_\varepsilon,
        \end{equation*}
        where $\nabla^2_\theta \coloneqq \grad_\theta \grad_\theta^\transp$ is the matrix with $(i,j)$ entry $\partial_{ij}$,
        and 
        \begin{equation*}
            A'_\varepsilon \coloneqq \Set{ \omega\in \Omega | \hat{\theta}_\varepsilon \in \mathring{\Theta} }
          \quad \text{and} \quad
          A''_\varepsilon \coloneqq \Set{ \omega\in \Omega | \abs{\det J_\varepsilon} \neq 0}.
        \end{equation*}
      \end{notation}
      
      \begin{lemma}\label{lem:twice_nabla}
        Under \Cref{assump:MuIsC^2}, $\nabla^2_\theta \Phi(\theta_0)=-\Sigma$.
      \end{lemma}
      
      \begin{proof}
        This is an immediate consequence from 
        \Cref{lem:measure_partial_change}.
      \end{proof}
      
      \begin{lemma}\label{lem:twice_nabla_uni_conv}
        Under \Cref{assump:MuIsC^2,assump:Theta}, 
        \begin{equation*}
          \sup_{\theta\in\Theta} \norm{ \nabla^2_\theta \qty{ \Phi_\varepsilon(\theta) - \Phi(\theta) } }_F \asconv 0
          \quad \text{as} \ 
          \varepsilon \to 0,
        \end{equation*}
        where $\norm{\cdot}_F$ is the Frobenius norm for $p\times p$ matrices.
      \end{lemma}

      \begin{proof}
        By \Cref{assump:Theta,assump:MuIsC^2},
        $\displaystyle\max_{i,j=1,\dots,p} \sup_{\theta\in\Theta}|\partial_i \partial_j \mu_\theta|$ is bounded. 
        It follows by  \Cref{lem:measure_partial_change} that
        \begin{equation*}
          \sup_{\theta\in\Theta} \norm{ \nabla^2_\theta \qty{ \Phi_\varepsilon(\theta) - \Phi(\theta) } }_F
          =\sup_{\theta\in\Theta} \norm{ \nabla^2_\theta \innerp{ \mu_\theta , \varepsilon Z } }_F 
          \leq \varepsilon \CkThMN[2]{\mu_\bullet} \SN{Z}
        \end{equation*}
        vanishes almost surely as $\varepsilon\to 0$.
      \end{proof}

      \begin{lemma}\label{lem:J_epsilon_convergence}
        Under \Cref{assump:Theta,assump:cont,assump:convex_Theta,assump:MuIsC^2,assump:isolated}, $J_\varepsilon \pconv -\Sigma$ as $\varepsilon \to 0$.
      \end{lemma}
      
      \begin{proof}
        Let $\theta^u_\varepsilon \coloneqq \theta_0+u(\hat{\theta}_\varepsilon-\theta_0), u\in[0,1]$  
        and take $r>0$ arbitrarily.
        Since $\Phi(\theta)$ is twice continuously differentiable, there exists $\delta>0$ such that 
        \begin{equation*}
          \sup_{|\theta-\theta_0|<\delta} \norm{ \nabla^2_\theta \Phi(\theta)-\nabla^2_\theta \Phi(\theta_0) }_F 
          < \frac{r}{3}.
        \end{equation*}
        Then, it holds that
        \begin{equation*}
            \begin{aligned}
              \norm{ J_\varepsilon-(-\Sigma) }_F
              &= \norm{ \int^1_0 \nabla^2_\theta \Phi_\varepsilon (\theta^u_\varepsilon) \dd{u} - \nabla^2_\theta \Phi(\theta_0) }_F \\
              &\le \sup_{\theta\in\Theta} \norm{ \nabla^2_\theta \qty{ \Phi_\varepsilon(\theta) - \Phi(\theta) } }_F
                + \int^1_0 \norm{ \nabla^2_\theta \qty{ \Phi(\theta^u_\varepsilon) - \Phi(\theta_0) } }_F \dd{u}
                \, \idv_{\{ |\theta^u_\varepsilon-\theta_0| \ge \delta \}}\\
              &\quad 
                +\int^1_0 \norm{ \nabla^2_\theta \Phi(\theta^u_\varepsilon) 
                -\nabla^2_\theta \Phi(\theta_0) }_F \dd{u}
                \, \idv_{\{|\theta^u_\varepsilon-\theta_0|<\delta\}}.
            \end{aligned}
        \end{equation*}
        From these inequalities, we see that
        \begin{align*}
          \Prob \qty( \norm{ J_\varepsilon - (-\Sigma) }_F > r )
          &\le 
            \Prob \qty(\sup_{\theta\in\Theta} \norm{ \nabla^2_\theta \qty{ \Phi_\varepsilon(\theta) - \Phi(\theta) } }_F > \frac{r}{3} ) \\
          &\quad+
            \Prob \qty( \int^1_0 \norm{ \nabla^2_\theta \qty{ \Phi(\theta^u_\varepsilon) - \Phi(\theta_0) } }_F \dd{u}
            \, \idv_{\{|\theta^u_\varepsilon-\theta_0|\ge\delta\}} > \frac{r}{3} ) \\
          & \quad
            + \Prob \qty( \int^1_0 \norm{ \nabla^2_\theta \Phi(\theta^u_\varepsilon) 
            - \nabla^2_\theta \Phi(\theta_0) }_F \dd{u}
            \, \idv_{\{|\theta^u_\varepsilon-\theta_0|<\delta\}} > \frac{r}{3} )\\
          &\le 
            \Prob \qty( \sup_{\theta\in\Theta} \norm{ \nabla^2_\theta \qty{ \Phi_\varepsilon(\theta) - \Phi(\theta) } }_F > \frac{r}{3} )
            + \Prob(|\theta^u_\varepsilon-\theta_0|\ge \delta).
        \end{align*}
        Thus, by \Cref{thm:parametric consistency} and \Cref{lem:twice_nabla_uni_conv}, we obtain the conclusion. 
      \end{proof}
      
      \begin{lemma}\label{lem:A_epsilon_convergence}
        Under \Cref{assump:Theta,assump:cont,assump:convex_Theta,assump:MuIsC^2,assump:Sigma_regular,assump:isolated}, 
        $P(A_\varepsilon) \to 1$ as $\varepsilon\to0$.
      \end{lemma}
      
      \begin{proof}
        Take sufficiently small $\delta$ so that the ball of radius $\delta$ centered at $\theta_0$ is contained in $\Theta$.
        Then, it follows by \Cref{thm:parametric consistency} that
        \begin{equation*}
          \Prob \qty\big( (A_\varepsilon')^\complement ) 
          \leq \Prob \qty( \abs\big{ \hat{\theta}_\varepsilon -\theta_0 } > \delta ) \to 0,\quad \varepsilon\to 0.
        \end{equation*}
        Since $\det J_\varepsilon \pconv \det(-\Sigma)$,
        \begin{equation*}
          \Prob \qty\big( (A_\varepsilon'')^\complement )
          \leq \Prob \qty\bigg( \abs\Big{ \det J_\varepsilon - \det(-\Sigma) } \geq \abs{ \det \Sigma } ) \to 0,\quad \varepsilon\to 0.
        \end{equation*}
        Thus, 
        \begin{equation*}
          \Prob(A_\varepsilon^\complement)
          \leq \Prob \qty\big( (A_\varepsilon')^\complement )
          + \Prob \qty\big( (A_\varepsilon'')^\complement )
          \to 0,
          \ \  \text{i.e.}, \ \ 
          \Prob(A_\varepsilon) \to 1,
        \end{equation*}
        as $\varepsilon\to0$.
      \end{proof}
      
      \begin{lemma}\label{lem:inverse_convergence}
        Under \Cref{assump:Theta,assump:cont,assump:convex_Theta,assump:MuIsC^2,assump:Sigma_regular,assump:isolated}, 
        $J^{-1}_\varepsilon \idv_{A_\varepsilon} \pconv -\Sigma^{-1}$ as $\varepsilon\to0$.
      \end{lemma}
      
      \begin{proof}
        Let $F:\Matp\to\Matp$ be given by
        \begin{equation*}
          F(A) \coloneqq 
          \begin{cases}
            A^{-1} & \text{if $A$ is invertible}, \\
            0      & \text{otherwise}.
          \end{cases}
        \end{equation*}
        Since $\det$ is a continuous map from $\Matp$ to $\R$, Cramer's rule shows that $F$ is continuous at every regular matrix.
        By \Cref{lem:J_epsilon_convergence} and the continuous mapping theorem with $\det(-\Sigma) \neq 0$, $F(J_\varepsilon) \pconv -\Sigma^{-1}$. 
        Thus, it follows from \Cref{lem:A_epsilon_convergence} that
        \begin{equation*}
          J_\varepsilon^{-1} \idv_{A_\varepsilon}
          = \Phi(J_\varepsilon) \idv_{A_\varepsilon}
          \pconv - \Sigma^{-1},
        \end{equation*}
        as $\varepsilon\to0$.
      \end{proof}

      \begin{proof}[Proof of \Cref{thm:parametric asymptotic normality}]
        By the mean value theorem, 
        \begin{equation*}
          \grad_\theta \Phi_\varepsilon (\hat{\theta}_\varepsilon) - \grad_\theta \Phi_\varepsilon (\theta_0)
          =\int^1_0 \nabla^2_\theta \Phi_\varepsilon \qty( (1-u) \theta_0 + u \hat{\theta}_\varepsilon ) \dd{u} 
          \, ( \hat{\theta}_\varepsilon-\theta_0 ).
        \end{equation*}
        By multiplying this by $\idv_{A_\varepsilon}$,
        \begin{equation*}
          ( \grad_\theta \Phi_\varepsilon ( \hat{\theta}_\varepsilon ) - \grad_\theta \Phi_\varepsilon ( \theta_0 ) ) \idv_{A_\varepsilon}
          = J_\varepsilon \idv_{A_\varepsilon} ( \hat{\theta}_\varepsilon-\theta_0 ).
        \end{equation*} 
        When $A_\varepsilon$ occurs, $J_\varepsilon$ is regular and $\hat{\theta}_\varepsilon$ is an interior maximum point of $\Phi_\varepsilon$, so that
        \begin{equation*}
          ( \hat{\theta}_\varepsilon - \theta_0 ) \idv_{A_\varepsilon}
          = -J_\varepsilon^{-1}
          \grad_\theta \Phi_\varepsilon ( \theta_0 ) \idv_{A_\varepsilon}.
        \end{equation*}
        According to \Cref{lem:A_epsilon_convergence}, it holds that 
        \begin{align*}
          \varepsilon^{-1} ( \hat{\theta}_\varepsilon - \theta_0 )
          &= -\varepsilon^{-1}J_\varepsilon^{-1} \grad_\theta \Phi_\varepsilon \qty( \theta_0 ) \idv_{A_\varepsilon}
          + \varepsilon^{-1} (\hat{\theta}_\varepsilon - \theta_0 ) \idv_{A_\varepsilon^\complement} \\
          &= -\varepsilon^{-1} J_\varepsilon^{-1} \grad_\theta \Phi_\varepsilon \qty(\theta_0) \idv_{A_\varepsilon}
          + o_p(1),
        \end{align*}
        Since $\grad_\theta \Phi(\theta_0)=0$ by \Cref{prop:ContinuousContrast_UniqueMaximum},
        it follows by
        \Cref{lem:measure_partial_change,lem:inverse_convergence} that
        \begin{equation*}
          \varepsilon^{-1} J_\varepsilon^{-1} \grad_\theta \Phi_\varepsilon \qty(\theta_0) \idv_{A_\varepsilon} 
          = J_\varepsilon^{-1} \innerp{ \grad_\theta \mu_{\theta_0}, Z } 
              \idv_{A_\varepsilon} 
          \pconv -\Sigma^{-1} \innerp{ \grad_\theta \mu_{\theta_0}, Z },
        \end{equation*}
        as $\varepsilon\to0$. Hence,
        \begin{equation*}
          \varepsilon^{-1} ( \hat{\theta}_\varepsilon - \theta_0 ) 
          \pconv \Sigma^{-1} \innerp*{ \grad_\theta \mu_{\theta_0}, Z },
          \quad 
          \varepsilon \to 0.
        \end{equation*}  
        The right-hand side follows $N(0,\Sigma^{-1})$.
        Indeed, 
        since $\grad_\theta \mu_{\theta_0} \in \MM^p$ 
        and $Z$ is a Gaussian vector on $\CC$,  
        $\innerp{ \grad_\theta \mu_{\theta_0}, Z }$
        is a centered Gaussian with covariance matrix 
        \begin{equation*}
          \qty( \E \qty[ \innerp{ \partial_i \mu_{\theta_0}, Z } \innerp{ \partial_j \mu_{\theta_0}, Z } ] )_{ij}
          = \qty( \innerp{ \partial_i \mu_{\theta_0}, \K (\partial_j \mu_{\theta_0}) } )_{ij}
          = \Sigma,
        \end{equation*}
        and so, the covariance matrix of $\Sigma^{-1} \innerp{ \grad_\theta \mu_{\theta_0}, Z }$ is $\Sigma^{-1}$.
      \end{proof}

    \subsubsection{Local Asymptotic Normality and Asymptotic Efficiency}

      In \Cref{subsec:conti_consistency,subsec:conti_AN}, we have shown that the estimator has consistency and asymptotic normality. In this section, we further show that this estimator is asymptotically efficient. We firstly show that the estimator is locally asymptotically normal at the true parameter.

      In this section, we write $P_\theta^\varepsilon$ as $P_{h_{\theta}}^\varepsilon$ for simplicity.

      \begin{theorem}[Local asymptotic normality]
        \label{thm:LAN1}

        Under \Cref{assump:Theta,assump:cont,assump:isolated,assump:convex_Theta,assump:MuIsC^2,assump:Sigma_regular},
        the family $\{ P_{\hat{\theta}_{\varepsilon}}^\varepsilon\}_{\varepsilon>0} \;$ is locally asymptotically normal at $\theta_0$.
      \end{theorem}
        
      \begin{proof}
        Note that $\Sigma$ is positive definite, because $\Sigma$ is regular and is the covariance matrix of 
        $\innerp{ \grad_\theta \mu_\theta , Z }$, 
        Therefore, we can define $\varphi(\varepsilon) \coloneqq \varepsilon \Sigma^{-1/2}$.
        It follows from the Cameron--Martin theorem (see, e.g., Theorem 9.3 in \textcite{lifshits1995gasussian}) that
        \begin{equation*}
            \log\dv{P_\theta^\varepsilon}{P_{\theta_0}^\varepsilon} \qty(X^\varepsilon)
            = \frac1{\varepsilon^2} \innerp{\mu_\theta-\mu_{\theta_0},X^\varepsilon-\K\mu_{\theta_0}}
            -\frac1{2\varepsilon^2} \innerp{\mu_\theta-\mu_{\theta_0},\K(\mu_\theta-\mu_{\theta_0})}
        \end{equation*}
        and so for any $u\in \mathbb{R}^p$ that
        \begin{equation*}
            \begin{aligned}
              \log\dv{P_{\theta_0+\varphi(\varepsilon) u}^\varepsilon}{P_{\theta_0}^\varepsilon}\qty(X^\varepsilon)
              &= \innerp*{ \frac{\mu_{\theta_0+\varphi(\varepsilon)u}- \mu_{\theta_0}}{\varepsilon}, Z }
              -\frac12 \innerp*{ \frac{\mu_{\theta_0+\varphi(\varepsilon)u}- \mu_{\theta_0}}{\varepsilon}, \K \left( 
              \frac{\mu_{\theta_0+\varphi(\varepsilon)u}- \mu_{\theta_0}}{\varepsilon} \right) } \\
              &\eqqcolon m^Z_\varepsilon(u) - \frac12 V_\varepsilon(u).
            \end{aligned}
        \end{equation*}
        Since the map $ \mu \mapsto \innerp{ \mu, \K\mu }$ is continuous and $\mu_\theta$ is continuously differentiable at $\theta_0$,
        \begin{equation*}
            \begin{aligned}
              \lim_{\varepsilon \to 0}V_\varepsilon(u)
              =& \lim_{\varepsilon \to 0} \innerp*{ 
                \frac{\mu_{\theta_0+\varepsilon\Sigma^{-1/2} u}- \mu_{\theta_0}}{\varepsilon}, 
                \K \qty( \frac{\mu_{\theta_0+\varepsilon\Sigma^{-1/2}u}- \mu_{\theta_0}}{\varepsilon} ) } \\
              =& \innerp*{ 
                \left(\Sigma^{-1/2} u\right)^\transp \grad_\theta \mu_{\theta_0},
                \K \qty( \qty( \Sigma^{-1/2} u )^\transp \grad_\theta \mu_{\theta_0} ) } 
              = (\Sigma^{-1/2} u)^\transp \Sigma (\Sigma^{-1/2} u)
              = |u|^2.
            \end{aligned}
        \end{equation*}
        Similarly, it follows for almost all $\omega\in\Omega$ that 
        \begin{equation*}
          \lim_{\varepsilon \to 0} m^Z_\varepsilon(u)
          =u^\transp \Sigma^{-1/2} \innerp{ \grad_\theta \mu_{\theta_0}, Z }
          \quad \text{a.s.}
        \end{equation*}
        Thus,
        \begin{equation*}
          \log \dv{P_{\theta_0+\varphi(\varepsilon) u}^\varepsilon}{P_{\theta_0}^\varepsilon}\qty(X^\varepsilon)
          = u^\transp \Delta_{ \varepsilon, \theta_0 } - \frac12 |u|^2 +o_p(1)
          \quad \text{as} \  \varepsilon \to 0,
        \end{equation*}
        where $\Delta_{ \varepsilon, \theta_0 } \coloneqq \Sigma^{-1/2} \innerp{ \grad_\theta \mu_{\theta_0}, \frac{X^\varepsilon -h_0}{\varepsilon} } \sim N(0,I_p)$ with the $p\times p$ identity matrix $I_p$.
      \end{proof}

      \begin{remark}[Asymptotic efficiency]
        According to \Cref{thm:parametric asymptotic normality}, we see that 
        \begin{equation*}
          \varphi(\varepsilon)^{-1} (\hat{\theta}_\varepsilon -\theta_0)=\Delta_{\varepsilon, \theta_0} + o_p(1),\quad \varepsilon\to 0, 
        \end{equation*}
        that is, $\{\hat{\theta}_\varepsilon\}$ is a sequence of {\it asymptotically centering estimators} at $\theta_0$; see \textcite{jeganathan1982} , Definition 2. Hence, under the LAN condition at $\theta_0$, it is {\it regular} in the sense of \textcite{ibragimov1981statistical}, Definition 9.1; see also \textcite{jeganathan1982}, Theorem 2. Then, \textcite{ibragimov1981statistical}, Theorem 9.1, so-called the {\it convolution theorem}, yields that $\hat{\theta}_\varepsilon$ is {\it asymptotically efficient} because it attains the minimum asymptotic variance. 
      \end{remark}

  \subsection{M-estimation under discrete observations}\label{sec:discrete}

    \subsubsection{Construction of estimators and consistency}\label{subsec:discrete_consistency}

      In this section, we consider the situation where we have the ``discrete" observation $X^{\varepsilon,n}$ with a known covariance operator $\K$ (resp. a covariance function $K$) which is defined as in \Cref{sec:functional} to estimate an unknown mean function $h_0$. The discrete observation $X^{\varepsilon,n}$ is defined as follows:

      \begin{notation}\label{note:discrete} Let $n\in \N$.
        \begin{itemize}
          \item $0=t_0<t_1<\cdots <t_n=T\;$ such that
                $\displaystyle\lim_{n\to\infty} \sup_{i=1,\dots,n} |t_i-t_{i-1}|=0$.   
          \item $X^{\varepsilon,n}\coloneqq (X^\varepsilon_{t_1},X^\varepsilon_{t_2},\dots,X^\varepsilon_{t_n})$.
          \item $T_1\coloneqq[0,t_1]$ and $T_i\coloneqq(t_{i-1},t_i]$ for $i=2,\dots,n$.
          \item $\mu^n_\theta = \sum_{i=1}^n \mu_\theta(T_i)\delta_{t_i} \in \MM$.
        \end{itemize}
      \end{notation}
    
      Similarly to the continuous observation case, 
      we impose \Cref{assump:Theta,assump:cont,assump:isolated}.
      By \Cref{assump:cont}, the function $\theta\mapsto \mu_\theta(T_i)$ is continuous,
      so the contrast function 
      \begin{align*}
        \Phi_{n,\varepsilon}(\theta) 
        &\coloneqq \innerp{ \mu^n_\theta, X^\varepsilon } - \frac12 \innerp{ \mu_\theta, \K^n\mu_\theta } \\
        &= \sum_{i=1}^n \mu_\theta(T_i) \, X^\varepsilon_{t_i} - \frac{1}{2}\sum_{i,j=1}^n \mu_\theta(T_i) \, \mu_\theta(T_j) \, K(t_i,t_j)
      \end{align*}
      is also continuous with respect to $\theta$ for each $n\in \N, \varepsilon>0$ .
      Therefore, due to the maximum principle, the estimators the following estimators $\hat{\theta}$, $\hat{\mu}_{n,\varepsilon}$ and $\hat{h}_{n,\varepsilon}$ are well-defined: 
      \begin{equation*}
        \hat{\theta}_{n,\varepsilon} 
        \coloneqq
        \argmax_{\theta\in \Theta } \Phi_{n,\varepsilon}(\theta), \quad
        \hat{\mu}_{n,\varepsilon} \coloneqq \mu_{\hat{\theta}_{n,\varepsilon}},\quad 
        \hat{h}_{n,\varepsilon} \coloneqq \K\hat{\mu}_{n,\varepsilon}.
      \end{equation*}
    
      \begin{theorem}[Consistency]\label{thm:dis_parametric_consistency}
        Under \Cref{assump:Theta,assump:cont,assump:isolated},
        the following convergences hold:
        \begin{equation*}
            \hat{\theta}_{n,\varepsilon} \asconv \theta_0 \ \  \text{in} \  \mathbb{R}^p,
            \quad 
            \hat{\mu}_{n,\varepsilon} \asconv \mu_{\theta_0} \ \  \text{in} \  \MM,
            \quad 
            \hat{h}_{n,\varepsilon} \asconv h_0 \ \  \text{in} \  \CC \  \text{and} \  \HP
        \end{equation*}
        as $\varepsilon\to 0 $ and $n\to \infty$.
      \end{theorem}   

      \begin{proof}
        It follows by definition of $\hat{\theta}_{n,\varepsilon}$ that $\Phi_{n,\varepsilon}(\hat{\theta}_{n,\varepsilon})\ge \Phi_{n,\varepsilon}(\theta_0)$. 
        From \Cref{prop:M-estimator,prop:ContinuousContrast_UniqueMaximum}, it is sufficient to show that
        $\sup_{\theta\in \Theta} |\Phi_{n,\varepsilon}(\theta)-\Phi(\theta)|\asconv 0$ as $\varepsilon\to0$ and $ n \to \infty$. 
        It follows by \Cref{assump:Theta,assump:cont} that
        \begin{equation*}
            \begin{aligned}
              \abs{ \Phi_{n,\varepsilon}(\theta) - \Phi_\varepsilon(\theta) }
              &= \abs{ \innerp{ \mu^n_\theta - \mu_\theta, h_0 + \varepsilon Z } 
              - \frac12 \innerp{ \mu_\theta, \left(\K^n-\K\right)\mu_\theta } } \\
              &\le 2 \varepsilon \CkThMN{\mu_\bullet} \SN{Z} 
              + \abs{ \innerp{\mu^n_\theta -\mu_\theta, \K \mu_0 } }
              + \frac{1}{2} \OPN{\K^n-\K}\CkThMN{\mu_\bullet}^2,
            \end{aligned}
            \end{equation*}
        and it is analogous to the derivation of \Cref{eq:func_conti_contrast_uni_conv} that
        \begin{equation*}
          \sup_{\theta \in \Theta} \abs{ \Phi_\varepsilon(\theta)- \Phi(\theta) }
          = \varepsilon \sup_{\theta \in \Theta} \abs{ \innerp{ \mu_{\theta}, Z } }
          \asconv  0
        \end{equation*}
        as $\varepsilon \to 0$.
        Thus, 
        from \Cref{lem:convergence of discret.mu} we obtain
        \begin{equation*}
          \sup_{\theta\in\Theta} \abs{ \Phi_{n,\varepsilon}(\theta) - \Phi(\theta) } 
          \le \sup_{\theta\in\Theta} \abs{ \Phi_{n,\varepsilon}(\theta) - \Phi_\varepsilon(\theta) }
          + \sup_{\theta\in\Theta} \abs{ \Phi_\varepsilon(\theta)-\Phi(\theta) }
          \asconv 0,
        \end{equation*}
        as $\varepsilon \to 0,\;n \to \infty$.
      \end{proof}
    
    \subsubsection{Asymptotic Normality}\label{subsubsec:AsympNormAsympEffi}

      This section shows the asymptotic properties of estimators from discrete observations. 
      First, we impose an {\it ad hoc} assumption for the convergence $K^n\to K$ stated in \Cref{assump:convergent_rate}, 
      which is essential to deduce the asymptotic normality of our estimators. 
      This assumption can be reduced to a simple condition that $n\varepsilon\to 0$ in a typical situation as described in \Cref{cor:convergent_rate_n_epsilon}, below.

      \begin{theorem}[Asymptotic normality]\label{thm:dis_parametric asymptotic normality}
        Under  \Cref{assump:Theta,assump:cont,assump:isolated,assump:convex_Theta,assump:MuIsC^2,assump:Sigma_regular,assump:convergent_rate}, it holds that 
        \begin{equation*}
          \varepsilon^{-1}(\hat{\theta}_{n,\varepsilon} - \theta_0) \pconv 
          \Sigma^{-1} \innerp{ \grad_\theta \mu_{\theta_0}, Z } 
          \quad \text{as} \ 
          \varepsilon\to 0, \  n\to \infty.
        \end{equation*}
      \end{theorem}
      Before the proof, 
      we prepare some notations and lemmas.
      
      \begin{notation}
        By \Cref{assump:MuIsC^2}, $\Phi_{n,\varepsilon}:\Theta\to \R$ is twice differentiable,  
        and thus we put
        \begin{equation*}
            J_{n,\varepsilon}
                \coloneqq \int^1_0 \nabla^2_\theta \Phi_{n,\varepsilon} \qty( (1-u) \theta_0 + u \hat{\theta}_{n,\varepsilon} )\dd{u}
            \quad \text{and} \quad 
            A_{n,\varepsilon}
                \coloneqq A^{'}_{n,\varepsilon} \cap A^{''}_{n,\varepsilon},
        \end{equation*}
        where
        \begin{equation*}
            A^{'}_{n,\varepsilon} 
                \coloneqq \Set{ \omega\in \Omega | \hat{\theta}_{n,\varepsilon} \in \mathring{\Theta} }
            \quad \text{and} \quad
            A^{''}_{n,\varepsilon}
                \coloneqq \Set{ \omega\in \Omega | \abs{ \det J_{n,\varepsilon} } \neq 0 }.
        \end{equation*}
      \end{notation}
      
      \begin{lemma}\label{lem:discrete_twice_nabla_uni_conv}
        Under \Cref{assump:Theta,assump:cont,assump:MuIsC^2},
        \begin{equation*}
          \sup_{\theta\in\Theta} \norm{ \nabla^2_\theta \qty{ \Phi_{n,\varepsilon}(\theta) - \Phi(\theta) } }_F \asconv  0
          \quad \text{as} \ 
          \varepsilon \to 0, \  n\to\infty,
        \end{equation*}
        where $\norm{\cdot}_F$ is the Frobenius norm for $p\times p$ matrices.
      \end{lemma}
      
      \begin{proof}
        By \Cref{lem:measure_partial_change,lem:convergence of discret.mu,,assump:MuIsC^2}, it follows for any $i,j=1,\dots,p$ that
        \begin{align*}
          \sup_{\theta\in\Theta} 
            \abs{ \partial_{ij} \qty{ \Phi_{n,\varepsilon}(\theta) - \Phi(\theta) } } 
          &= \sup_{\theta\in\Theta}
            \left| \varepsilon\innerp{(\partial_{ij} \mu_\theta)^n,Z} + \innerp{\partial_{ij}(\mu_\theta^n -\mu_\theta), \K\mu_0}\right. \\
          &\quad
            \left. - \innerp{\partial_{ij}\mu_\theta, (\K^n-\K) \mu_\theta} 
            - \innerp{\partial_{i}\mu_\theta, (\K^n-\K) \qty(\partial_j\mu_\theta)} \right|\\
          &\le \sup_{\theta\in\Theta} 
            \Big\{ \varepsilon\MN{(\partial_{ij} \mu_\theta)^n}\SN{Z} + \innerp{(\partial_{ij}\mu_\theta)^n -\partial_{ij}\mu_\theta, \K\mu_0}\\
          &\quad
            +\OPN{\K^n-\K}\left(\MN{\partial_{ij}\mu_\theta}\MN{\mu_\theta} + \MN{\partial_{i}\mu_\theta}\MN{\partial_j\mu_\theta}\right) \Big\}. 
        \end{align*}
        Therefore, we obtain the conclusion from \Cref{assump:Theta,assump:cont,lem:convergence of discret.mu}.
      \end{proof}
      
      The following three lemmas are analogous to   \Cref{lem:J_epsilon_convergence,lem:A_epsilon_convergence,lem:inverse_convergence}, respectively.
    
      \begin{lemma}\label{lem:dis_J_epsilon_convergence}
        Under \Cref{assump:cont,assump:convex_Theta,assump:isolated,assump:MuIsC^2,assump:Theta}, $J_{n,\varepsilon} \pconv -\Sigma$ as $\varepsilon \to 0$, \  $n\to\infty$.
      \end{lemma}
      
      \begin{lemma}\label{lem:dis_A_epsilon_convergence}
        Under \Cref{assump:cont,assump:convex_Theta,assump:isolated,assump:MuIsC^2,assump:Theta,assump:Sigma_regular},
        $P(A_{n,\varepsilon}) \to 1$ as $\varepsilon\to 0$, $n\to \infty$.
      \end{lemma}
      
      \begin{lemma}\label{lem:dis_inverse_convergence}
        Under \Cref{assump:cont,assump:convex_Theta,assump:isolated,assump:MuIsC^2,assump:Theta,assump:Sigma_regular},
        $J^{-1}_{n,\varepsilon}\idv_{A_{n,\varepsilon}} 
          \pconv - \Sigma^{-1}$ as $\varepsilon \to 0$, $n \to \infty$.
      \end{lemma}
            
      \begin{lemma}\label{lem:ConvergentOrder}
        Under \Cref{assump:convex_Theta,assump:MuIsC^2,assump:Theta,assump:convergent_rate},
        \begin{equation*}
          \varepsilon^{-1} \abs{\grad_\theta \Phi_{n,\varepsilon}(\theta_0) - \grad_\theta \Phi_\varepsilon(\theta_0) } 
          \pconv 0, \quad \varepsilon\to 0, \  n\to \infty.
        \end{equation*}
      \end{lemma}
      
      \begin{proof}
        As in the proof of \Cref{thm:dis_parametric_consistency}, we have
        \begin{equation*}
            \begin{aligned}
              \varepsilon^{-1} \abs{ \partial_i\qty( \Phi_{n,\varepsilon}(\theta)-\Phi_\varepsilon(\theta)) }
              &= \varepsilon^{-1} \abs{ \partial_i \innerp{ \mu^n_\theta -\mu_\theta, h_0+\varepsilon Z } 
              -\frac{1}{2} \partial_i \innerp{ \mu_\theta, \left(\K^n-\K\right)\mu_\theta } } \\
              &\le \abs{ \innerp{ (\partial_i\mu_\theta)^n - \partial_i\mu_\theta, Z} }
              + \varepsilon^{-1} \abs{ \innerp{\partial_i\mu_\theta, (\K\mu_0)^n - \K\mu_0 } } \\
              &\quad+ \frac{1}{2} \varepsilon^{-1}\|K^n-K\|_{L^\infty}\CkThMN{\mu_\bullet}^2. 
            \end{aligned}
        \end{equation*}
        It follows by \Cref{assump:convergent_rate} that
        \begin{equation}
          \begin{aligned}
            &\varepsilon^{-1}\|(\K\mu_0)^n-\K\mu_0\|_{L^\infty} \\
            &\quad 
              \le \varepsilon^{-1} \sup_{t\in T} \abs{ \qty( \int_T K(s,t) \, \mu_0(\dd{s}) )^n - \int_T K(s,t)\mu_0(\dd{s})} \\
            &\quad 
              \le \varepsilon^{-1} \sup_{t\in T} \sum_{i,j=1}^n \int_T \qty(\abs{K(s,t_i)-K(t_i,t_j)} 
               + \abs{K(t_i,t_j)-K(s,t)} ) \, \idv_{T_i\times T_j}(t,s) \, \mu_0(\dd{s})\\
            &\quad 
              \le 2 \varepsilon^{-1}\norm{K^n - K}_{L^\infty} \MN{\mu_0}
            \to 0,
          \end{aligned}
          \label{eq:NormConvergenceOfMeanByCovSmoothness}
        \end{equation}
        as $\varepsilon\to 0$, $n\to\infty$.
        Thus, by \Cref{assump:convex_Theta,assump:MuIsC^2,assump:Theta}, the conclusion holds.
      \end{proof}
      
      \begin{proof}[Proof of \Cref{thm:dis_parametric asymptotic normality}]
        Similarly to the proof of \Cref{thm:parametric asymptotic normality}, 
        \begin{equation}
            \begin{aligned}
                \varepsilon^{-1} \qty( \hat{\theta}_{n,\varepsilon} - \theta_0 )
                &= \varepsilon^{-1} \qty( \hat{\theta}_{n,\varepsilon}-\theta_0 ) \idv_{A_{n,\varepsilon}}
                + \varepsilon^{-1} \qty( \hat{\theta}_{n,\varepsilon} - \theta_0 ) \idv_{A_{n,\varepsilon}^\complement} \\
                &= -\varepsilon^{-1} J_{n,\varepsilon}^{-1} \grad_\theta \Phi_{n,\varepsilon}(\theta_0)\idv_{A_{n,\varepsilon}}
                + o_p(1).
            \end{aligned}
          \label{eq:dis_asymptotic_1}
        \end{equation}
        Since $\grad_\theta \Phi(\theta_0)=0$ and $\grad_\theta \Phi_\varepsilon(\theta_0)=\varepsilon\innerp{\grad \mu_{\theta_0},Z}$ by \Cref{prop:ContinuousContrast_UniqueMaximum}, it follows by  \Cref{lem:measure_partial_change,lem:dis_inverse_convergence,lem:ConvergentOrder} and by \Cref{assump:convergent_rate} that
        \begin{equation}
            \begin{aligned}
              \varepsilon^{-1} J_{n,\varepsilon}^{-1} \grad_\theta \Phi_{n,\varepsilon} (\theta_0) \idv_{A_{n,\varepsilon}}
              &= 
                \varepsilon^{-1} J_{n,\varepsilon}^{-1} \qty( \grad_\theta \Phi_{n,\varepsilon}(\theta_0) 
                - \grad_\theta \Phi_\varepsilon(\theta_0) ) \idv_{A_{n,\varepsilon}}
                + J_{n,\varepsilon}^{-1} \innerp*{ \grad_\theta \mu_{\theta_0}, Z } 
                \idv_{A_\varepsilon} \\
              &\pconv -\Sigma^{-1} \innerp*{ \grad_\theta \mu_{\theta_0}, Z }
            \end{aligned}
          \label{eq:dis_Z1_convergence}
        \end{equation}
        as $n\to \infty$ and $\varepsilon\to 0$.
        Therefore, it follows by \Cref{eq:dis_asymptotic_1,eq:dis_Z1_convergence} that
        \begin{equation*}
            \varepsilon^{-1} \left( \hat{\theta}_{n,\varepsilon} -\theta_0 \right) \pconv
            \Sigma^{-1} \innerp*{ \grad_\theta \mu_{\theta_0}, Z }
        \end{equation*}
        as $n\to \infty$ and $\varepsilon\to 0$.
      \end{proof}
      
      \begin{remark}
        Since the asymptotic variance of the estimator from discrete observations is the same as the one from continuous observations, we can conclude that our estimator is also asymptotically efficient in the sense of minimum asymptotic variance.
      \end{remark}
      
      \begin{corollary}\label{cor:convergent_rate_n_epsilon}
        Assume 
        \begin{equation}\label{eq:t_order}
          \sup_{i=1,\dots,n} \abs{ t_i - t_{i-1} } = \order{\frac1n},
        \end{equation} 
        and $K$ be Lipschitz continuous on $T^2$.
        Then, under \Cref{assump:Theta,assump:cont,assump:isolated,assump:convex_Theta,assump:MuIsC^2,assump:Sigma_regular},
        \begin{equation*}    
          \varepsilon^{-1}\qty( \hat{\theta}_{n,\varepsilon}-\theta_0 ) \pconv \Sigma^{-1} \innerp{ \grad_\theta \mu_{\theta_0}, Z },
        \end{equation*}
        as $n\to\infty$, $\varepsilon\to0$ and $n\varepsilon\to \infty $. 
      \end{corollary}
      
      \begin{proof}
        By \Cref{thm:dis_parametric asymptotic normality}, it suffices to prove that the convergent rate \Cref{assump:convergent_rate} is satisfied. 
        It follows from \Cref{lem:measure_partial_change,eq:t_order} that
        \begin{equation*}
          \varepsilon^{-1}\|K^n - K\|_{L^\infty}
          \le\varepsilon^{-1} \max_{s\in T_i, t\in T_j} \abs{K(t_i, t_j)- K(s,t)}
          =\order{\frac{1}{n\varepsilon}}
        \end{equation*}
        as $n\varepsilon\to\infty$, and so, it completes the proof.
      \end{proof}

    \subsubsection{Moment Convergence}\label{sec:MomentConvergence}

      In this section, we show the moment convergence of $\hat{\theta}_{n,\varepsilon}$.
      In a statistical context of this issue, there is a celebrated result brought in the far-famed method initiated by \textcite{yoshida2011polynomial} which uses the \emph{polynomial type large deviation inequality} for the likelihood ratio random fields. 
      However, our situation is much simpler: we focus on a Gaussian case, and our contrast function is formally quadratic. In such a case, we can take an alternative root to get the moment convergence results.

      The crux of our approach is to use the \emph{energy inequality method}, which is well-known in literatures of PDEs (cf. \textcite{evans2010partial}, \textcite{ladyzenskaja1968linear} and among others), 
      and the highlight of our argument is to obtain an \emph{energy upper bound} in \Cref{prop:EnergyForMomentConv}.

      We establish at the onset the following \emph{coercive property}, which corresponds to the identifiability condition for $\Phi(\theta)$ in the statistical point of view; see, e.g., \textcite{yoshida2011polynomial}, (A.3). 

      \begin{proposition}[Coercive property]\label{prop:coersive}
        Under \Cref{assump:Theta,assump:cont,assump:isolated,assump:convex_Theta,assump:MuIsC^2,assump:Sigma_regular},
        there is a positive constant $c>0$ such that
        \begin{equation*}
            \begin{aligned}
              \innerp*{\mu_\theta - \mu_{\theta_0}, \K(\mu_\theta - \mu_{\theta_0})}
              &\geq c \abs{ \theta - \theta_0}^2, \\
              \innerp*{\mu_\theta - \mu_{\theta_0}, \K(\mu_\theta - \mu_{\theta_0})}
              &\geq c \MN{ \mu_\theta - \mu_{\theta_0}}^2
            \end{aligned}
        \end{equation*}
        for all $\theta\in\Theta$.
      \end{proposition}

      \begin{proof}
        From Taylar's formula,
        \begin{equation*}
            \begin{aligned}
              \innerp*{\mu_\theta - \mu_{\theta_0}, \K(\mu_\theta - \mu_{\theta_0})}
              &= (\theta-\theta_0)^\transp \Sigma (\theta-\theta_0)
                + o(\abs{ \theta - \theta_0 }^2) \\
              &\geq \lambda_{\Sigma}^{\mathsf{min}} 
                  \abs{ \theta - \theta_0 }^2
                + o(\abs{ \theta - \theta_0 }^2)
                \quad \text{as} \ \ 
                \theta \to \theta_0,
            \end{aligned}
        \end{equation*}
        where $\lambda_{\Sigma}^{\mathsf{min}}>0$ is the minimum eigenvalue of $\Sigma$. 
        There is some $\delta>0$ such that
        \begin{equation*}
          \innerp*{\mu_\theta - \mu_{\theta_0}, 
              \K(\mu_\theta - \mu_{\theta_0})}
          \geq \frac{ \lambda_{\Sigma}^{\mathsf{min}} }{ 2 } 
              \abs{ \theta - \theta_0 }^2
          \quad \text{when} \ \ 
          \abs{ \theta - \theta_0 } < \delta.
        \end{equation*}
        Since $\innerp*{\mu_\theta - \mu_{\theta_0}, \K(\mu_\theta - \mu_{\theta_0})} = 0$ if and only if $\theta=\theta_0$ by \Cref{assump:isolated} 
        and $\Theta$ is compact, we obtain that 
        \begin{equation*}
          \inf_{\substack{\abs{ \theta - \theta_0 } \geq \delta \\ \theta \in \Theta}}
          \frac{ \innerp*{\mu_\theta - \mu_{\theta_0}, \K(\mu_\theta - \mu_{\theta_0})} }{ \abs{ \theta - \theta_0 }^2 } > 0,
        \end{equation*}
        which is the first inequality.
        Next, we have 
        \begin{equation*}
          \MN{ \mu_\theta - \mu_{\theta_0}}^2 \leq \lambda_{\Sigma}^{\mathsf{max}} \abs{\theta - \theta_0}^2 + o(\abs{ \theta - \theta_0 }^2 )
        \end{equation*}
        and by the same argument as above, we obtain the second inequality.
      \end{proof}

      Apparently, this result is helpful for us not only to make convergence theorems for $\hat{\theta}_{n,\varepsilon}$,
      but also
      to elaborate an \emph{energy upper bound} for $\hat{\mu}_{n,\varepsilon}$
      when we prove the following proposition.

      \begin{proposition}[Energy upper bounds]\label{prop:EnergyForMomentConv}
        Under the same assumptions as in \Cref{prop:coersive},
        \begin{equation}\label{ineq:EnergyForMomentConv3}
          \begin{aligned}
              \frac{c_1}{\varepsilon^2}
              \innerp*{ \mu_{ \hat{\theta}_{n,\varepsilon}} - \mu_{ \theta_0 } , \K ( \mu_{ \hat{\theta}_{n,\varepsilon}} - \mu_{ \theta_0 } ) } 
            & \leq 
              \frac1{\varepsilon^2}
              \norm{K^n - K}_{L^\infty}^2
              + \abs{ \innerp*{ \grad_\theta \mu_{ \theta_0 } , Z } }^2 
              + \sup_{\theta\in\Theta} \norm{ \innerp*{ \nabla_\theta^2 \mu_\theta , Z } }_F^2 \\
            &\quad 
              + \abs{ \innerp*{ \grad_\theta ( \mu^n - \mu )_{ \theta_0 } , Z } }^2
              + \sup_{\theta\in\Theta} \norm{ \innerp*{ \nabla_\theta^2 ( \mu^n - \mu )_\theta , Z } }_F^2,
          \end{aligned}
        \end{equation}
        where the positive constants $c$
        depends only on $\K$, $\theta_0$, $\{\mu_\theta\}_{\theta\in\Theta}$ and $\Theta$.
      \end{proposition}

      \begin{proof}
        An immediate consequence from $\hat{\theta}_{n,\varepsilon} \in \argmax \Phi_{n,\varepsilon}$ is that
        \begin{align*}
          0 
          &\leq \Phi_{n,\varepsilon} ( \hat{\theta}_{n,\varepsilon} ) 
            - \Phi_{n,\varepsilon} ( \theta_0 ) \\
          &= 
            \varepsilon \innerp*{ \mu_{ \hat{\theta}_{n,\varepsilon}}^n - \mu_{ \theta_0 }^n , Z } 
            - \frac12 \innerp*{ \mu_{ \hat{\theta}_{n,\varepsilon}}^n - \mu_{ \theta_0 } , \K ( \mu_{ \hat{\theta}_{n,\varepsilon}}^n - \mu_{ \theta_0 } ) }
            + \frac12 \innerp*{ \mu_{ \theta_0 }^n - \mu_{ \theta_0 } , \K ( \mu_{ \theta_0 }^n - \mu_{ \theta_0 } ) }
            \\
          &= \varepsilon \innerp*{ \mu_{ \hat{\theta}_{n,\varepsilon}}^n - \mu_{ \theta_0 }^n , Z } 
            - \frac12 \innerp*{ \mu_{ \hat{\theta}_{n,\varepsilon}} - \mu_{ \theta_0 } , \K ( \mu_{ \hat{\theta}_{n,\varepsilon}} - \mu_{ \theta_0 } ) }
            + \frac12 \innerp*{ \mu_{ \theta_0 }^n - \mu_{ \theta_0 } , \K ( \mu_{ \theta_0 }^n - \mu_{ \theta_0 } ) } \\
          &\quad 
            - \frac12 \innerp*{ \mu_{ \hat{\theta}_{n,\varepsilon} }^n - \mu_{ \hat{\theta}_{n,\varepsilon}} , \K ( \mu_{ \hat{\theta}_{n,\varepsilon} }^n - \mu_{ \hat{\theta}_{n,\varepsilon}} ) }
            - \innerp*{ \mu_{ \hat{\theta}_{n,\varepsilon} }^n - \mu_{ \hat{\theta}_{n,\varepsilon}} , \K ( \mu_{ \hat{\theta}_{n,\varepsilon} } - \mu_{ \theta_0 } ) },
        \end{align*}
        and so we obtain
        \begin{equation}\label{ineq:EnergyForMomentConv1}
          \begin{aligned}
            &\innerp*{ \mu_{ \hat{\theta}_{n,\varepsilon}} - \mu_{ \theta_0 } , \K ( \mu_{ \hat{\theta}_{n,\varepsilon}} - \mu_{ \theta_0 } ) } \\
            &\quad
              \leq 2 \varepsilon \innerp*{ \mu_{ \hat{\theta}_{n,\varepsilon}}^n - \mu_{ \theta_0 }^n , Z } 
              -2 \innerp*{ \mu_{ \hat{\theta}_{n,\varepsilon} }^n - \mu_{ \hat{\theta}_{n,\varepsilon}} , \K ( \mu_{ \hat{\theta}_{n,\varepsilon} } - \mu_{ \theta_0 } ) } \\
            &\qquad
              + \innerp*{ \mu_{ \theta_0 }^n - \mu_{ \theta_0 } , \K ( \mu_{ \theta_0 }^n - \mu_{ \theta_0 } ) }
              - \innerp*{ \mu_{ \hat{\theta}_{n,\varepsilon} }^n - \mu_{ \hat{\theta}_{n,\varepsilon}} , \K ( \mu_{ \hat{\theta}_{n,\varepsilon} }^n - \mu_{ \hat{\theta}_{n,\varepsilon}} ) }.
          \end{aligned}
        \end{equation}
        in order to simplify this energy inequality,
        by using Young's inequality for products and \Cref{prop:coersive},
        we estimate from above the second term and the last two terms in \Cref{ineq:EnergyForMomentConv1} as follows:
        \begin{align*}
          &\abs{ \innerp*{ \mu_{ \hat{\theta}_{n,\varepsilon} }^n - \mu_{ \hat{\theta}_{n,\varepsilon}} , \K ( \mu_{ \hat{\theta}_{n,\varepsilon} } - \mu_{ \theta_0 } ) } } 
            \leq 
              \MN{ \mu_{ \hat{\theta}_{n,\varepsilon} } - \mu_{ \theta_0 } }
              \SN{ \K ( \mu_{ \hat{\theta}_{n,\varepsilon} }^n - \mu_{ \hat{\theta}_{n,\varepsilon}} ) } \\
          &\qquad
            \leq
              \frac{c}8 \MN{ \mu_{ \hat{\theta}_{n,\varepsilon} } - \mu_{ \theta_0 } }^2
              + \frac2c \SN{ \K ( \mu_{ \hat{\theta}_{n,\varepsilon} }^n - \mu_{ \hat{\theta}_{n,\varepsilon}} ) }^2 \\
          &\qquad
            \leq
              \frac18 \innerp*{ \mu_{ \hat{\theta}_{n,\varepsilon}} - \mu_{ \theta_0 } , \K ( \mu_{ \hat{\theta}_{n,\varepsilon}} - \mu_{ \theta_0 } ) }
              + \frac4c \norm{K^n - K}_{L^\infty}^2 
              \CkThMN{\mu_\bullet}^2,
        \end{align*}
        and
        \begin{align*}
          &\innerp*{ \mu_{ \theta_0 }^n - \mu_{ \theta_0 } , \K ( \mu_{ \theta_0 }^n - \mu_{ \theta_0 } ) }
            - \innerp*{ \mu_{ \hat{\theta}_{n,\varepsilon} }^n - \mu_{ \hat{\theta}_{n,\varepsilon}} , \K ( \mu_{ \hat{\theta}_{n,\varepsilon} }^n - \mu_{ \hat{\theta}_{n,\varepsilon}} ) } \\
          &\qquad
            = \innerp*{ \mu_{ \theta_0 }^n - \mu_{ \hat{\theta}_{n,\varepsilon} }^n - \mu_{ \theta_0 } + \mu_{ \hat{\theta}_{n,\varepsilon}}, 
            \K ( \mu_{ \theta_0 }^n - \mu_{ \theta_0 } ) 
            - \K ( \mu_{ \hat{\theta}_{n,\varepsilon} }^n - \mu_{ \hat{\theta}_{n,\varepsilon}} ) } \\
          &\qquad
            \leq 4 
              \MN{ \mu_{ \hat{\theta}_{n,\varepsilon} } - \mu_{ \theta_0 } }
              \sup_{\theta\in\Theta} \SN{ \K (\mu_{ \theta }^n - \mu_{ \theta } ) }\\
          &\qquad
            \leq 
              \frac14 \innerp*{ \mu_{ \hat{\theta}_{n,\varepsilon}} - \mu_{ \theta_0 } , \K ( \mu_{ \hat{\theta}_{n,\varepsilon}} - \mu_{ \theta_0 } ) }
              + \frac8c \norm{K^n - K}_{L^\infty}^2 
              \CkThMN{\mu_\bullet}^2,
        \end{align*}
        where the constant $c>0$ is the same as in \Cref{prop:coersive}.
        Combining these inequalities,
        the energy inequality \Cref{ineq:EnergyForMomentConv1} becomes
        \begin{equation}\label{ineq:EnergyForMomentConv2}
          \begin{aligned}
            &\innerp*{ \mu_{ \hat{\theta}_{n,\varepsilon}} - \mu_{ \theta_0 } , \K ( \mu_{ \hat{\theta}_{n,\varepsilon}} - \mu_{ \theta_0 } ) } 
              \leq \frac{32}c \norm{K^n - K}_{L^\infty}^2 
                \CkThMN{\mu_\bullet}^2
                + 4 \varepsilon \innerp*{ \mu_{ \hat{\theta}_{n,\varepsilon}}^n - \mu_{ \theta_0 }^n , Z }.
          \end{aligned}
        \end{equation}
        Our next aim is to thin out $\hat{\theta}_{n,\varepsilon}$ from the right-hand side of this energy inequality. 
        It follows from Taylor's formula and Young's inequality that
        \begin{equation}
          \begin{aligned}
              &\varepsilon^{-1} \innerp*{ \mu_{ \hat{\theta}_{n,\varepsilon}} - \mu_{ \theta_0 } , Z } \\
            &\quad = 
              \varepsilon^{-1} ( \hat{\theta}_{n,\varepsilon} -\theta_0)^\transp 
              \innerp*{ \grad_\theta \mu_{ \theta_0 } , Z }
            + 
              \varepsilon^{-1} ( \hat{\theta}_{n,\varepsilon} -\theta_0)^\transp 
              \int_0^1 \innerp*{ \nabla_\theta^2 \mu_{ u \hat{\theta}_{n,\varepsilon} + (1-u) \theta_0 } , Z } \dd{u} ( \hat{\theta}_{n,\varepsilon} -\theta_0) \\
            &\quad \leq 
              \frac{c}{16\varepsilon^2} \abs{ \hat{\theta}_{n,\varepsilon} - \theta_0 }^2 
              + \frac8c \abs{ \innerp*{ \grad_\theta \mu_{ \theta_0 } , Z } }^2 
              + \frac8c \sup_{\theta\in\Theta} \norm{ \innerp*{ \nabla_\theta^2 \mu_\theta , Z } }_F^2 \sup_{\theta\in\Theta} \abs{ \theta - \theta_0 }^2, 
          \end{aligned}
          \label{eq:ineq:EnergyForMomentConv3}
        \end{equation}
        and it also follows that
        \begin{equation*}
          \begin{aligned}
              & \varepsilon^{-1} \innerp*{ \mu_{ \hat{\theta}_{n,\varepsilon}}^n - \mu_{ \theta_0 }^n - \mu_{ \hat{\theta}_{n,\varepsilon}} + \mu_{ \theta_0 } , Z } \\
              & \quad = \varepsilon^{-1} ( \hat{\theta}_{n,\varepsilon} -\theta_0)^\transp 
              \innerp*{ \grad_\theta \mu_{ \theta_0 }^n - \grad_\theta \mu_{ \theta_0 } , Z } \\
            &\qquad + 
              \varepsilon^{-1} ( \hat{\theta}_{n,\varepsilon} -\theta_0)^\transp 
              \int_0^1 \innerp*{ \nabla_\theta^2 ( \mu^n - \mu )_{ u \theta + (1-u) \theta_0 } , Z } \dd{u} 
              ( \hat{\theta}_{n,\varepsilon} -\theta_0)\\
            & \quad \leq 
              \frac{c}{16\varepsilon^2} \abs{ \hat{\theta}_{n,\varepsilon} - \theta_0 }^2
              + \frac8c \abs{ \innerp*{ \grad_\theta ( \mu^n - \mu )_{ \theta_0 } , Z } }^2 
              + \frac8c \sup_{\theta\in\Theta} \norm{ \innerp*{ \nabla_\theta^2 ( \mu^n - \mu )_\theta , Z } }_F^2 
              \sup_{\theta\in\Theta} \abs{ \theta - \theta_0 }^2.
          \end{aligned}
        \end{equation*}
        Applying \Cref{prop:coersive} to these inequalities properly again,
        we can take away $\hat{\theta}_{n,\varepsilon}$ from \Cref{ineq:EnergyForMomentConv2}
        as mentioned above
        and finally obtain the following energy inequality 
        which provides us a mount of benefits later:
        \begin{equation*}
          \begin{aligned}
            &
              \frac{c}{64\varepsilon^2}
              \innerp*{ \mu_{ \hat{\theta}_{n,\varepsilon}} - \mu_{ \theta_0 } , \K ( \mu_{ \hat{\theta}_{n,\varepsilon}} - \mu_{ \theta_0 } ) } \\
            &\leq 
              \frac1{\varepsilon^2}
              \norm{K^n - K}_{L^\infty}^2 
              \CkThMN{\mu_\bullet}^2 \\
            &\quad + 
              \abs{ \innerp*{ \grad_\theta \mu_{ \theta_0 } , Z } }^2 
              + \sup_{\theta\in\Theta} \norm{ \innerp*{ \nabla_\theta^2 \mu_\theta , Z } }_F^2 \sup_{\theta\in\Theta} \abs{ \theta - \theta_0 }^2 \\
            &\quad 
              + \abs{ \innerp*{ \grad_\theta ( \mu^n - \mu )_{ \theta_0 } , Z } }^2
              + \sup_{\theta\in\Theta} \norm{ \innerp*{ \nabla_\theta^2 ( \mu^n - \mu )_\theta , Z } }_F^2 
              \sup_{\theta\in\Theta} \abs{ \theta - \theta_0 }^2.
          \end{aligned}
        \end{equation*}
        This completes the proof.
      \end{proof}

      \begin{remark}
        In this proof, we obtain a useful inequality \Cref{eq:ineq:EnergyForMomentConv3}
        as a by-product. 
        Although it is futile without next \Cref{prop:MomentConvergence},
        it plays an important role in the proof of \Cref{prop:ApplEnergyMethod}.
      \end{remark}

      In the following proposition and in the subsequent section, we additionally need \Cref{assump:SmoothBoundaryForMorrey,assump:MuIsC^3}, 
      since we use Morrey's inequality. 
    
      \begin{theorem}[Moment convergence]\label{prop:MomentConvergence}
        Under \Cref{assump:Theta,assump:cont,assump:isolated,assump:convex_Theta,assump:MuIsC^2,assump:Sigma_regular,assump:MuIsC^3,assump:SmoothBoundaryForMorrey,assump:convergent_rate},
        for any $q \in [1,\infty)$
        \begin{equation*}
          \begin{aligned}
            \varepsilon^{-1} ( \hat{\theta}_{n,\varepsilon} - \theta_0 )
            &\to \Sigma^{-1} \innerp*{ \grad_\theta \mu_{\theta_0}, Z } 
            & \quad &
            \text{in} \ \  L^q\qty\big((\Omega,\Prob);\mathbb{R}^p),
            \\
            \varepsilon^{-1} ( \mu_{ \hat{\theta}_{n,\varepsilon} } - \mu_{\theta_0} )
            &\to \grad_\theta \mu_{\theta_0}^\transp \Sigma^{-1} \innerp*{ \grad_\theta \mu_{\theta_0}, Z }
            & &
            \text{in} \ \  L^q\qty\big((\Omega,\Prob);\MM)
          \end{aligned}
        \end{equation*}
        as $\varepsilon\to0$ and $n\to\infty$.
      \end{theorem}

      \begin{proof}
        From \Cref{prop:EnergyForMomentConv,prop:coersive}, 
        we have the right-hand side of \Cref{ineq:EnergyForMomentConv3} as a dominating function of $( \hat{\theta}_{n,\varepsilon} - \theta_0 )/\varepsilon$ (or $( \mu_{ \hat{\theta}_{n,\varepsilon} } - \mu_{\theta_0} ) / \varepsilon$).
        To use the dominated convergence theorem later, we shall show that the dominating function
        \begin{equation*}
          \begin{aligned}
            &
              \frac1{\varepsilon^2}
              \norm{K^n - K}_{L^\infty}^2 
              + \abs{ \innerp*{ \grad_\theta \mu_{ \theta_0 } , Z } }^2 
              + \sup_{\theta\in\Theta} \norm{ \innerp*{ \nabla_\theta^2 \mu_\theta , Z } }_F^2 \\
            & \quad
              + \abs{ \innerp*{ \grad_\theta ( \mu^n - \mu )_{ \theta_0 } , Z } }^2
              + \sup_{\theta\in\Theta} \norm{ \innerp*{ \nabla_\theta^2 ( \mu^n - \mu )_\theta , Z } }_F^2, \label{D-func}
          \end{aligned}
        \end{equation*}
        converges in $L^q(\Omega,\Prob)$ as $\varepsilon\to0$ and $n\to\infty$.
        
        The first term clearly converges to $0$ by \Cref{assump:convergent_rate}. 
        From \Cref{prop:Morrey}, 
        \begin{align*}
          \E \CkN[2]{ \innerp*{ \mu_\bullet , Z }}^{2m}
          &\leq C \, \E \WkqN[3,2m]{ \innerp*{ \mu_\bullet , Z } }^{2m}
          \leq C \OPN{\K}^m \CkThMN[3]{\mu_\bullet}^{2m}, \\
          \E \CkN[2]{ \innerp*{ \mu_\bullet^n - \mu_\bullet , Z }}^{2m}
          &\leq C \, \E \WkqN[3,2m]{ \innerp*{ \mu_\bullet^n - \mu_\bullet , Z } }^{2m} \\
          &\leq C \norm{K^n-K}_{L^\infty}^m \CkThMN[3]{\mu_\bullet}^{2m}
        \end{align*}
        for some appropriate constants $C$. Therefore we see that the second and the third terms in \eqref{D-func} have any order of moments, 
        and the fourth and the fifth terms converge to 0 in $L^2(\Omega,\Prob)$.
        
        Hence, by the dominated convergence theorem and \Cref{thm:dis_parametric asymptotic normality}, 
        the desired convergence for $( \hat{\theta}_{n,\varepsilon} - \theta_0 )/\varepsilon$ and  $( \mu_{ \hat{\theta}_{n,\varepsilon} } - \mu_{\theta_0} ) / \varepsilon$ hold true.
      \end{proof}

    \subsubsection{Model Selection}\label{subsec:ModelSelection}

      This section considers the model selection for the parameterized mean function of a Gaussian process and constructs AIC-type information criteria. That is, we use our contrast function (quasi-log likelihood) as a substitution of the log-likelihood ratio and construct the bias-adjusted estimator for its limit function. It is often called ``contrast-based information criteria". 
      \Cref{thm:QGAIC} enables us to consider that our construction of new information criteria is legitimate. 
      
      The way to prove \Cref{thm:QGAIC} in this section is based on the energy inequality which is elaborated in the preceding section, and is more sophisticated and quicker than the usual way (see  \Cref{sec:AnotherProofModelSelection} to compare the standard argument for the proof of \Cref{thm:QGAIC}).
      The fact that the proof in this section is brief and simple strengthens how the energy method is powerful and elegant.
      
      \begin{theorem}[Bias for contrast function]\label{thm:QGAIC}
        Suppose \Cref{assump:Theta,assump:cont,assump:isolated,assump:convex_Theta,assump:Sigma_regular,assump:MuIsC^3,assump:SmoothBoundaryForMorrey,assump:convergent_rate,assump:MuIsC^2} and that 
        \begin{equation}\label{eq:ConvergenceRateForAIC}
          \varepsilon^{-2}\|K^n -K\|_{L^\infty}\to 0,\quad\varepsilon\to 0, \  n\to \infty.
        \end{equation}
        Then it follows that 
        \begin{equation*}
          \E\left[\Phi_{n,\varepsilon}(\hat{\theta}_{n,\varepsilon}) 
          - \Phi(\hat{\theta}_{n,\varepsilon})\right] 
          = p\varepsilon^2 +  o(\varepsilon^2),\quad\varepsilon\to 0, \  n\to\infty.
        \end{equation*}
      \end{theorem}
          
      \begin{remark}
        The additional assumption \Cref{eq:ConvergenceRateForAIC}, which is stronger than \Cref{assump:convergent_rate},
        will be needed to estimate the second and the third terms of the right-hand side of \Cref{eq:AnotherProofQGAIC1},
        and also needed to obtain \Cref{eq:AIC_D2} in the proof of \Cref{thm:QGAIC}.
        We commonly use this assumption in the two different ways of the proof, one is the following and the other is given in \Cref{sec:AnotherProofModelSelection}.
      \end{remark}
      
      \begin{lemma}\label{prop:ApplEnergyMethod}
      Under \Cref{assump:Theta,assump:cont,assump:isolated,assump:convex_Theta,assump:Sigma_regular,assump:MuIsC^3,assump:SmoothBoundaryForMorrey,assump:convergent_rate,assump:MuIsC^2},
      \begin{equation*}
        \E \qty[ \frac{1}{\varepsilon} \innerp*{ \mu_{ \hat{\theta}_{n,\varepsilon}}, Z } ]
        \to p
      \end{equation*}
      as $\varepsilon\to0$ and $n\to\infty$.
      \end{lemma}
      
      \begin{proof}
        In the proof of \Cref{prop:EnergyForMomentConv},
        we have shown that
        \begin{equation*}
          \frac{c}{\varepsilon} \innerp*{ \mu_{ \hat{\theta}_{n,\varepsilon}} - \mu_{ \theta_0 } , Z } 
            \leq 
              \frac{1}{\varepsilon^2} \abs{ \hat{\theta}_{n,\varepsilon} - \theta_0 }^2 
              + \abs{ \innerp*{ \grad_\theta \mu_{ \theta_0 } , Z } }^2 
              + \sup_{\theta\in\Theta} \norm{ \innerp*{ \nabla_\theta^2 \mu_\theta , Z } }_F^2, 
        \end{equation*}
        where the constant $c>0$ 
        depends only on $\K$, $\theta_0$, $\{\mu_\theta\}_{\theta\in\Theta}$ and $\Theta$.
        From this and \Cref{prop:MomentConvergence},
        we obtain an $L^1$-convergent dominating function of $\varepsilon^{-1} \innerp{ \mu_{ \hat{\theta}_{n,\varepsilon}} - \mu_0, Z }$, and so by \Cref{prop:MomentConvergence} again we obtain that 
        \begin{align*}
          \E \qty[ \frac{1}{\varepsilon} \innerp*{ \mu_{ \hat{\theta}_{n,\varepsilon}} - \mu_0, Z } ]
          &\to \E \qty[ \innerp*{ \grad_\theta \mu_{\theta_0}^\transp \Sigma^{-1} \innerp*{ \grad_\theta \mu_{\theta_0}, Z }, Z} ] \\
          &= \E \qty[ \innerp*{ \grad_\theta \mu_{\theta_0}^\transp, Z} \Sigma^{-1} \innerp*{ \grad_\theta \mu_{\theta_0}, Z } ] \\
          &= \tr \qty( \Sigma^{-1} \Sigma ) = p
        \end{align*}
        as $\varepsilon\to0$ and $n\to\infty$. 
        Since $\E[\innerp{\mu_0,Z}]=0$, the conseequence follows.
      \end{proof}
            
      \begin{proof}[Proof of \Cref{thm:QGAIC}]
        From the definition of $\Phi_{n,\varepsilon}$ and $\Phi$,
        \begin{equation}\label{eq:AnotherProofQGAIC1}
          \Phi_{n,\varepsilon}(\theta) - \Phi(\theta)
          = \varepsilon \innerp*{ \mu_\theta^n, Z} - \frac12 \innerp*{\mu_\theta, (\K^n-\K) \mu_\theta} + \innerp*{\mu_{\theta_0}, \K (\mu_\theta^n - \mu_\theta)},
        \end{equation}
        and the second and the third term in the right-hand side is dominated by
        \begin{equation*}
          \norm{K^n-K}_{L^\infty} \CkThMN[1]{\mu_\bullet}^2
        \end{equation*}
        up to constants. 
        Moreover, as in the proof of \Cref{prop:MomentConvergence}, we have that 
        \begin{equation*}
          \E \CkN{ \innerp*{ \mu_\bullet^n - \mu_\bullet , Z }}
          \leq C \norm{K^n-K}_{L^\infty}^{1/2} \CkThMN[1]{\mu_\bullet}
        \end{equation*}
        with an appropriate constant $C$.
        Thus, by \Cref{eq:ConvergenceRateForAIC,prop:ApplEnergyMethod},
        \begin{align*}
          &\frac1{\varepsilon^2} \E \qty[ \Phi_{n,\varepsilon}(\hat{\theta}_{n,\varepsilon}) - \Phi(\hat{\theta}_{n,\varepsilon}) ] \\
          &\quad= \E \qty[ \frac1\varepsilon \innerp*{ \mu_{\hat{\theta}_{n,\varepsilon}}, Z} ] 
          + \frac1\varepsilon \E \qty[ \innerp*{ \mu_{\hat{\theta}_{n,\varepsilon}}^n - \mu_{\hat{\theta}_{n,\varepsilon}}, Z} ] \\
          &\qquad
          + \frac1{\varepsilon^2} \E \qty[
          - \frac12 \innerp*{\mu_{\hat{\theta}_{n,\varepsilon}}, (\K^n-\K) \mu_{\hat{\theta}_{n,\varepsilon}}} + \innerp*{\mu_{\theta_0}, \K (\mu_{\hat{\theta}_{n,\varepsilon}}^n - \mu_{\hat{\theta}_{n,\varepsilon}})} ],
        \end{align*}
        converges to $-p$ as $\varepsilon\to0$ and $n\to\infty$,
        and the proof is completed.    
      \end{proof}
      
      As the consequence of this section, with the aid of \Cref{thm:QGAIC}, we propose the following AIC-type information criteria.
      
      \begin{definition}
        The \emph{quasi-information criterion} for mean functions of  GPs (QGAIC) is defined as follows:
        \begin{equation*}
          \text{QGAIC} \coloneqq -2 \Phi_{n,\varepsilon}(\hat{\theta}_{n,\varepsilon}) + 2\varepsilon^2p.
        \end{equation*}
      \end{definition}
    
  \section{Examples and Numerical Results}\label{sec:Numerical}

    In this section, we present some simple examples of Gaussian processes and parameter settings for
    $\theta= (\theta_1,\theta_2,\dots,\theta_k)
    \in \R^k$. In each example, we also state what conditions are necessary for the example to satisfy \Cref{assump:Sigma_regular,assump:cont,assump:convex_Theta,assump:isolated,assump:Theta,assump:MuIsC^2,eq:t_order} with a Lipschitz continuity for $K$. In particular, \Cref{assump:Sigma_regular} is often important and difficult to confirm whether they are satisfied or not, therefore we will only review their implications in a few examples.
    
    In addition, in the case that the true centered Gaussian process $Z$ is Ornstein--Uhlenbeck (OU) process, we check the consistency and asymptotic normality of estimator by numerical simulation as the results of 
    \Cref{thm:dis_parametric_consistency,thm:dis_parametric asymptotic normality}. 
    In the numerical simulation, the parameter setting corresponds to 1 dimensional version of \cref{eq:linear_paramater} in Example \ref{ex:density}.

    \subsection{Examples}

      Firstly, we present two simple examples.

      \begin{example}[Linear model with Dirac's $\delta$ basis]\label{ex:Delta}
        In order to set the parametric model for the mean function, we put $\{\mu_\theta\}_{\theta\in\Theta}$ by using the Dirac measure as
        \begin{equation*}
            \mu_\theta\coloneqq \sum_{i=1}^p \theta_i\delta_{s_i},
            \quad 
            \theta=(\theta_1,\dots,\theta_p) \in \Theta,
        \end{equation*}
        where $s_i \in T$ ($i=1,\dots,p$) with $ s_1<s_2<\cdots<s_p$.
        Then, the parametric model $\{h_\theta\}_{\theta\in\Theta}$ and the contrast function $\Phi_{n,\varepsilon}(\theta)$ is given by
        \begin{equation*}
            h_\theta(t) 
                = \sum_{i=1}^p \theta_i K(s_i,t), \quad
            \Phi_{n,\varepsilon}(\theta) 
                = \sum_{i=1}^p \theta_i X^\varepsilon_{\ceil{s_i}}
                - \frac12 \sum_{i,j=1}^p \theta_i \theta_j K(\ceil{s_i},\ceil{s_j}),
        \end{equation*}
        where $\ceil{s_i} \coloneqq \min \{ t_j \,|\, t_j\geq s_i, j=1,\dots,n \}$.
      \end{example}

      \begin{remark}
          When $n=p$, this example has a relation to the kernel method.
          To explain more precisely, we consider the case in which  all the $s_i$ are appeared in the observation time $\{t_1,\dots,t_n\}$ (or continuous observation case).
          Then $\Phi_{n,\varepsilon}(\theta)$ is given by
          \begin{equation*}
              \Phi_{n,\varepsilon}(\theta) 
                = \sum_{i=1}^p \theta_i X^\varepsilon_{s_i}
                - \frac12 \sum_{i,j=1}^p \theta_i \theta_j K(s_i,s_j),
          \end{equation*}
          and our estimator is obtained as $\hat{\theta}_{n,\varepsilon}=\Sigma^{-1} b$ which amounts to
          \begin{equation*}
              \argmin_{\theta} \sum_{j=1}^p \abs{X^\varepsilon_{s_j} -\sum_{i=1}^p \theta_i K(s_i,s_j)}^2,
          \end{equation*}
          where $\Sigma$ is the matrix with $(i,j)$ entry $\Sigma_{ij}=K(s_i,s_j)$, and $b$ is the vector with $j$-th entry $b_j=X^\varepsilon_{s_j}$.
          Thus, under this particular situation, our estimator corresponds to the estimator obtained by the kernel method without penalty term,
          but unfortunately, the statistical situations and the asymptotic theories for them are quite different; one is the case in which $p$ is fixed and we estimate the mean function $h_{\theta_0}$ under $n\to\infty$, $\varepsilon\to0$, and the other is the case in which $\varepsilon$ is fixed and we interpolate the value $X^\varepsilon_t$ at unobserved time $t$ under $p=n\to\infty$. 
      \end{remark}

      \begin{example}[Measures with the densities]\label{ex:density}
        We put $\{\mu_\theta\}_{\theta\in\Theta}$ as
        \begin{equation*}
            \mu_\theta(\dd{t})\coloneqq f_\theta(t)\dd{t},
        \end{equation*}
        where $f_\theta\in \Ck[2]{\Theta;L^1(T)}$.
        Then,
        \begin{equation*}
            \begin{aligned}
                h_\theta(t) 
                &= \int^T_0 K(s,t) \, f_\theta (s) \dd{s}, \\
                \Phi_{n,\varepsilon} (\theta)
                &= \sum_{i=1}^n X^\varepsilon_{t_i} \int_{T_i} f_\theta(s) \dd{s} 
                - \frac{1}{2} \sum_{i,j=1}^n K(t_i,t_j) \int_{T_i} f_\theta(s) \dd{s} \int_{T_j} f_\theta(t) \dd{t}.
            \end{aligned}
        \end{equation*}
        and the covariance matrix of the limit distribution for $\varepsilon^{-1}(\hat{\theta}_{n,\varepsilon}-\theta_0)$ is given by the matrix $\Sigma$ with $(i,j)$ entry
        \begin{equation*}
          \Sigma_{ij} = \iint_{T^2} K(s,t) \,
          \pdv{f_{\theta_0}}{\theta_i}\qty(s) \, \pdv{f_{\theta_0}}{\theta_j}\qty(t) 
          \dd{s} \dd{t}.
        \end{equation*}
        In addition, if $f_\theta$ is of the form
        \begin{equation}\label{eq:linear_paramater}
          f_{\theta}(t) = \sum_{k=1}^p \theta_k f_{k}(t),
        \end{equation}
        for $f_k \in L^2(T)$ ($k=1,\dots,p$), then the $(i,j)$ entry of the limit covariance matrix $\Sigma$ is given by
        \begin{equation*}
          \Sigma_{ij} = \iint_{T^2} K(s,t) \, f_i(s) \, f_j(t) \dd{s} \dd{t},
        \end{equation*}
        and our estimator $\hat{\theta}_{n,\varepsilon}$ is given by $\hat{\theta}_{n,\varepsilon} = (\Sigma^n)^{-1} b$,
        where the vector $b$ and the matrix $\Sigma^n$ are given by
        \begin{equation*}
          b_k = \sum_{i=1}^n X^\varepsilon_{t_i}\int_{T_i} f_k(t) \dd{t}, \quad
          \Sigma^n_{k\ell} = \sum_{i,j=1}^n K(t_i,t_j) \int_{T_i}f_k(t) \dd{t} \int_{T_j}f_l(s) \dd{s}.
        \end{equation*}
        We shall use these expressions for the computation in our numerical experiment.
      \end{example}

      Next, we present two major examples of GPs.

      \begin{example}[Markovian GPs]\label{ex:Markov}\leavevmode \par
        \begin{itemize}
          \item Let $u,v:T\to\R$ be bounded functions, and suppose that $v$ is non-decreasing.
          \item $Z \sim N(0,\K)$, where $K(s,t)\coloneqq u(s)u(t)\min\{v(s), v(t)\}$.
        \end{itemize}
      \end{example} 

      \Cref{ex:Markov} contains various major Gaussian processes, including the following processes:
      \begin{itemize}
        \item \emph{Wiener process} : $K(s,t) = \min(s,t)$,
        \item \emph{Brownian bridge} : $K(s,t) = \min(s,t) - st = \min(s,t)(1- \max(s,t))$,
        \item \emph{OU process} : $K(s,t) = \sigma^2/2\eta
        \qty{ \exp(-\eta|s-t|) - \exp(-\eta(s+t)) }$, for $\eta, \sigma>0$.
      \end{itemize}

      \begin{example}[Fractional Brownian motion (fBm)]\label{ex:fractional}\leavevmode \par
        \begin{itemize}
          \item $Z\sim N(0,\K)$, where the covariance function is of the form
          \begin{equation*}
            K(s,t) = |s|^{2H} + |t|^{2H} - |t-s|^{2H}
          \end{equation*}
          for $0<H<1$.
        \end{itemize}
      \end{example}

      \begin{remark}
        The covariance function in \Cref{ex:fractional}
        is Lipschitz continuous on $T^2$, \Cref{assump:convergent_rate} can be replaced by simply  $n\varepsilon\to \infty$. The covariance functions in \Cref{ex:Markov} can be Lipschitz continuous: e.g., the covatiance functions of Wiener process, Brownian bridge and OU process are Lipschitz continuous.)
      \end{remark}

    \subsection{Numerical results}

      In this section, we shall observe consistency and asymptotic normality for the case where the Gaussian process Z on $[0,1]$ is OU process with $\eta=\frac{1}{2},\ \sigma=1$ in \Cref{ex:Markov}. We consider the 1-dimensional case of \labelcref{eq:linear_paramater} in \Cref{ex:density}. 
      We consider the following situation:
      \begin{itemize}
        \item Let $T=1,\;\;t_i = \frac{i}{n}$ (observation times).
        \item Let $K(s,t) = \exp(-\frac{1}{2}|s-t|) - \exp(-\frac{1}{2}(s+t))$.
        \item Let $\mu_{\theta}(\dd{s}) = \theta\cdot (6\sin(-7.9s)) \dd{s}$.
        \item Let $h_\theta(t) = \theta \int_0^1 K(s,t) \cdot 6\sin(-7.9s) \dd{s}$.
        \item Let $\theta_0 = -4.0$ (true parameter).
        \item $X^\varepsilon\sim N(h_{\theta_0}, \varepsilon^2 \K)\ \ (\varepsilon>0)$.
        \item $\Sigma = 36 \int_0^1 \int_0^1 K(s,t) \sin(-7.9s) \sin(-7.9t) \dd{s} \dd{t} \approx 0.18077$.
        \item $\varepsilon^{-1}(\hat{\theta}_{n\varepsilon}-\theta_0)\dconv N(0,\sigma^2)$, $\sigma\approx2.3520$.
      \end{itemize}
      We observe a discrete sample $X^{\varepsilon,n}\coloneqq (X^\varepsilon_{t_1},X^\varepsilon_{t_2},\dots,X^\varepsilon_{t_n})$ from a one path of $X^\varepsilon$ on $[0,1]$.
      
      We compute the mean and the standard deviation of $\hat{\theta}_{n,\varepsilon}$ with 1000 iterations (see \Cref{table:WeakConsistncy}) under the following four cases: $(n, \varepsilon)=(100,0.1), (1000,0.1), (100,0.01)$ and $(1000,0.01)$.
      Samples for $\hat{h}_{n,\varepsilon}$ are shown in \Cref{fig:consistency}, 
      and we present
      QQ-plots and histgrams with respect to $\varepsilon^{-1} (\hat{\theta}_{n,\varepsilon} - \theta_0)$ with 1000 iterations (see \Cref{fig:AsymototicNormality,fig:QQPlot}).
      These numerical experiments, in particular the case $(n, \varepsilon)=(100, 0.01)$, indicate that if \Cref{assump:convergent_rate}, in particular $n\varepsilon\to0$ with a Lipschiz continuity for $K$, is not satisfied, 
      then
      the asymptotic normality for $\hat{\theta}_{n,\varepsilon}$ would fail. 

      \begin{table}[htbp]
        \centering
        \begin{tabular}{cccc}
          \hline
          $\varepsilon$& $n=100$ & $n=1000$ & TRUE \\
          \hline
          0.1 & -3.97473 & -4.00146 & -4.0 \\ [.1em]
            & (0.171050) & (0.188378) \\ [.4em] 
          0.01 & -3.97789 & -3.99837 & -4.0\\ [.1em]
            & (0.019060) & (0.018642) \\ 
          \hline
        \end{tabular}\caption{Means (standard deviations) of $\hat\theta_n$ with 1000 iterations.}
        \label{table:WeakConsistncy}
      \end{table}

      \begin{figure}[htbp]
        \centering
        \begin{minipage}[htbp]{0.4\hsize}
        \centering
        $n = 100$, $\varepsilon = 0.1$
        \includegraphics[width=0.85\hsize]{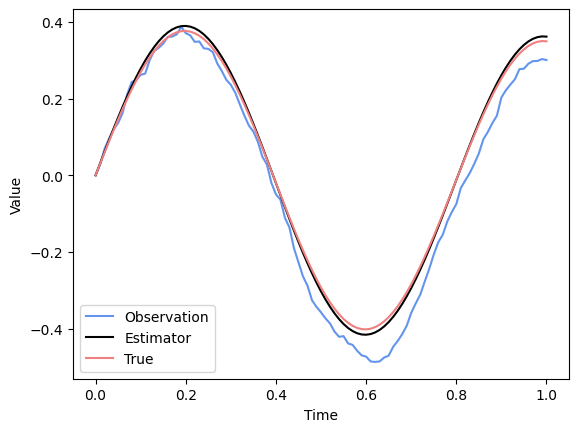}
  
        $n = 100$, $\varepsilon = 0.01$
        \includegraphics[width=0.85\hsize]{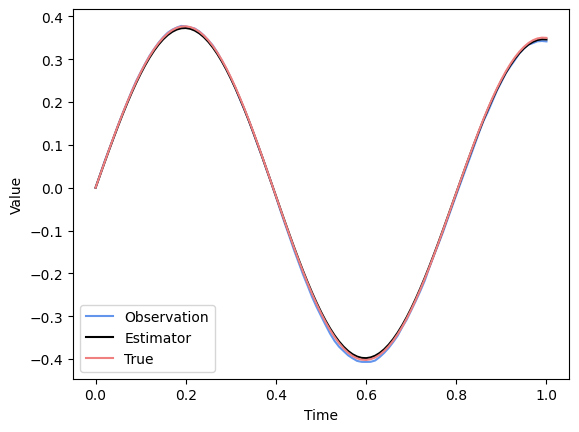}
      \end{minipage}
      \begin{minipage}[htbp]{0.4\hsize}
        \centering
        $n = 1000$, $\varepsilon = 0.1$
        \includegraphics[width=0.85\hsize]{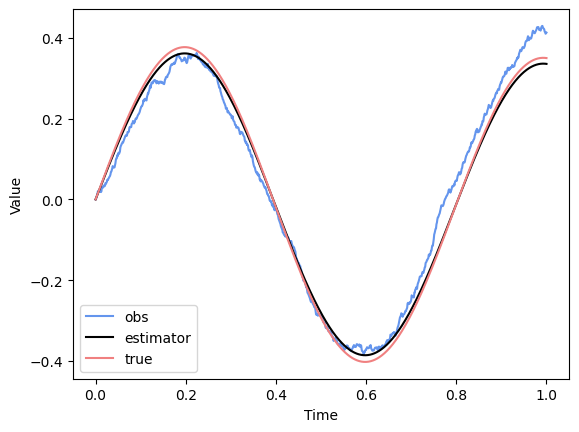}
        
        $n = 1000$, $\varepsilon = 0.01$
        \includegraphics[width=0.85\hsize]{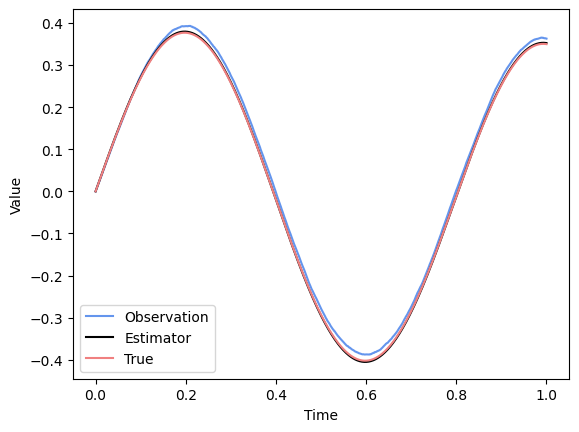}
      \end{minipage}
      \caption{Sample Paths of $X^\varepsilon$ (blue), the true mean functions (orange) 
      and the estimators $\hat{h}_{n,\varepsilon}$ (black).}\label{fig:consistency}
      \end{figure}
  
      \begin{figure}[htbp]
        \centering
        \begin{minipage}[htbp]{0.4\hsize}
          \centering
          $n = 100$, $\varepsilon = 0.1$
          \includegraphics[width=0.85\columnwidth]{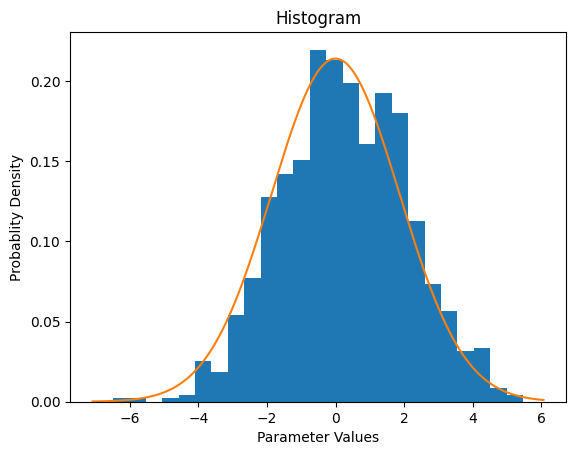}
          \vfill
          $n = 100$, $\varepsilon = 0.01$
          \includegraphics[width=0.85\columnwidth]{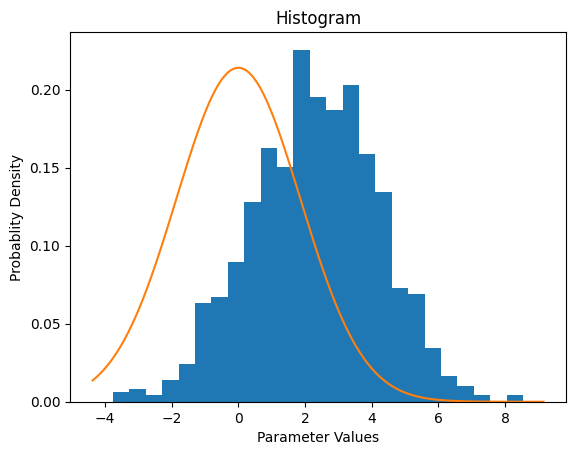}
        \end{minipage}
        \begin{minipage}[htbp]{0.4\hsize}
          \centering
          $n = 1000$, $\varepsilon = 0.1$
          \includegraphics[width=0.85\columnwidth]{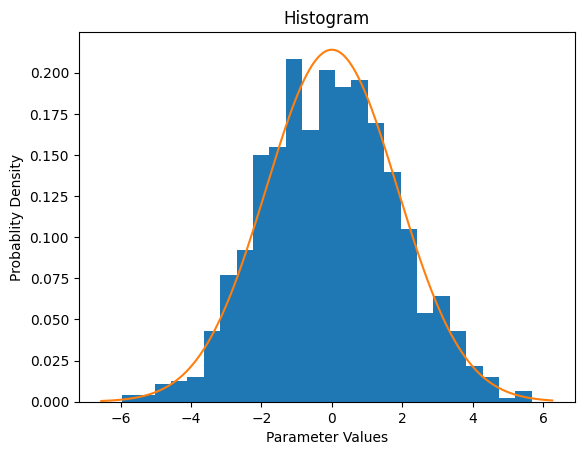}
          \vfill
          $n = 1000$, $\varepsilon = 0.01$
          \includegraphics[width=0.85\columnwidth]{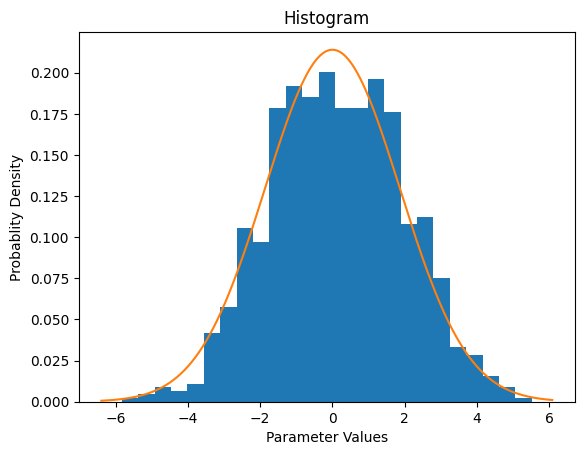}
        \end{minipage}
        \caption{Histograms for $\varepsilon^{-1}(\hat{\theta}_{n,\varepsilon}-\theta_0)$ with 1000 iterations, and the density of $N(0,2.35^2)$ (orange) with 1000 iterations.}\label{fig:AsymototicNormality}
      \end{figure}
  
      \begin{figure}[htbp]
        \centering
        \begin{minipage}[htbp]{0.40\hsize}
          \centering
          $n = 100$, $\varepsilon = 0.1$
          \includegraphics[width=0.85\columnwidth]{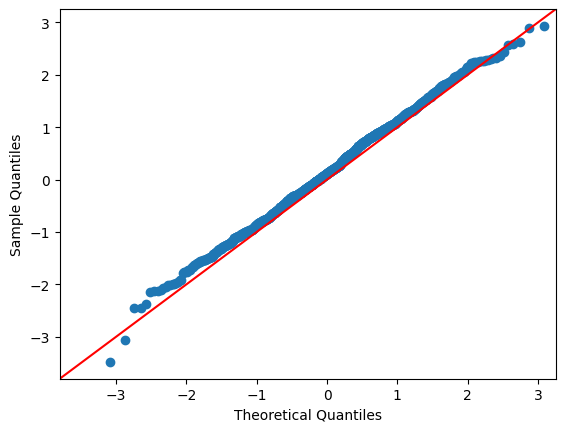}
          \vfill
          \vfill
          $n = 100$, $\varepsilon = 0.01$
          \vfill
          \vfill
          \includegraphics[width=0.85\columnwidth]{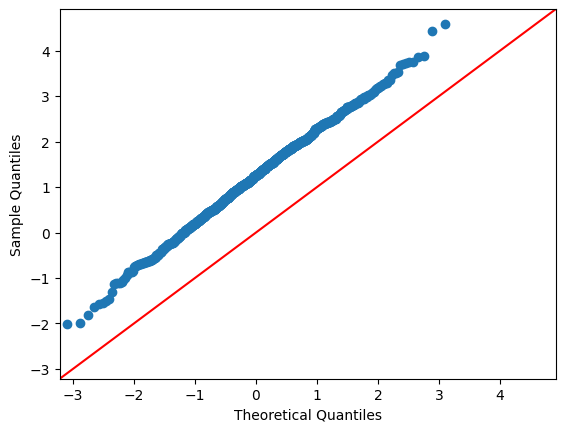}
        \end{minipage}
        \begin{minipage}[htbp]{0.40\hsize}
          \centering
          $n = 1000$, $\varepsilon = 0.1$
          \includegraphics[width=0.85\columnwidth]{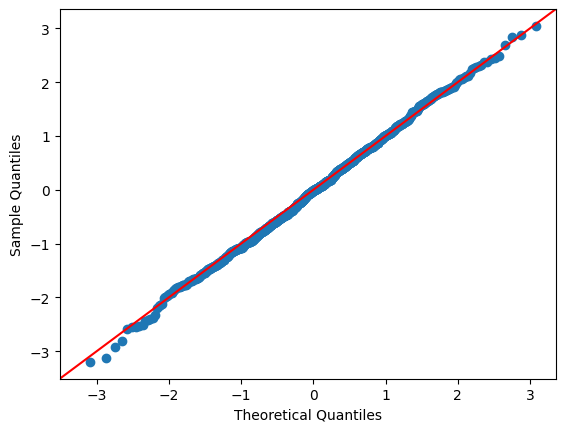}
          \vfill
          \vfill
          $n = 1000$, $\varepsilon = 0.01$
          \vfill
          \vfill
          \includegraphics[width=0.85\columnwidth]{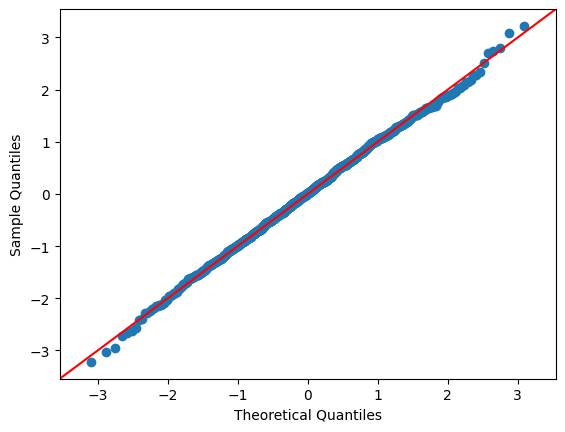}
        \end{minipage}
        \caption{QQ-Plots for $\varepsilon^{-1}(\hat{\theta}_{n,\varepsilon}-\theta_0)$ with 1000 iterations.}\label{fig:QQPlot}
      \end{figure}

  \section*{Acknowledgement}
      This work is partially supported by JSPS KAKENHI Grant-in-Aid for Scientific Research (C) \#24K06875, and Japan Science and Technology Agency CREST  JPMJCR2115. Also, the authors sincerely thank the anonymous reviewers for their insightful comments that have enhanced the quality of this paper. 

\appendix

\section*{Appendix}
  \addcontentsline{toc}{section}{Appendices}

\section{Operators}\label{apdx:operators}

  We shall give the propositions and lemmas,
  and the symbols $\X$, $\X^*$, $\K$,
  $\XP$, $I^*$, $I$, $\MMM$ and $\Theta$ are the same spaces and operators presented in \Cref{sec:prelimnaries,sec:mainresult}.
  \begin{proposition}\label{prop:OpnormOfIandI*}
    Let $\X$ be a Banach space.
    $I^*:\X^*\to\XP$ and $I:\XP\to\X$ are bounded, and their operator norms satisfy
    \begin{equation*}
      \OPN{I^*} = \OPN{I} = \OPN{\K}^{1/2}.
    \end{equation*}
  \end{proposition}

  \begin{proof}
    Recall that $\K$ is a bounded linear operator between Banach space $\X^*$ and $\X$. The map $I^*:\X^*\to\XP$ is a bounded linear operator with $\OPN{I^*}\leq\OPN{\K}^{1/2}$, since 
    \begin{equation*}
      \norm{I^*\mu}_{\XP}^2
      = \innerp{\mu,\K\mu}
      \leq \OPN{\K}\norm{\mu}^2.
    \end{equation*}
    Moreover, $I:\XP\to\X$ is also a bounded linear operator with $\OPN{I}\leq\OPN{\K}^{1/2}$, since
    \begin{equation*}
      \begin{aligned}
        \norm{Iz}_{\X}
        &= \sup_{\norm{\mu}\leq1} \innerp{\mu,Iz}
        = \sup_{\norm{\mu}\leq1} \qty( I^*\mu, z )_{\XP} \\
        &\leq \sup_{\norm{\mu}\leq1} \norm{I^*\mu}_{\XP} \norm{z}_{\XP}
        \leq \OPN{\K}^{1/2} \norm{z}_{\XP}
      \end{aligned}
    \end{equation*}
    for any $z \in \XP$. Finally, $\OPN{\K} \leq \OPN{I}\OPN{I^*}$ leads us the consequence.
  \end{proof}

  \begin{lemma}\label{lem:sym_schwartz}
    For $\mu,\nu \in \MM$, we have
    \begin{equation*}
        \SN{\K\mu} \le \OPN{\K} \MTN{[\mu]},
        \quad 
        |\innerp{ \mu, \K\nu }| \le \SN{\K \mu} \MTN{[\nu]}
        \quad \text{and} \quad 
        \norm{ \K \mu }_{\HP}^2
        \le \OPN{\K} \norm{[\mu]}^2_{\MMM}
    \end{equation*}
  \end{lemma}

  \begin{proof}
    By the definition of the quotient norm, 
    \begin{equation*}
        \begin{aligned}
          \SN{\K\mu}
          &= \inf_{\varphi\in N} \SN{ \K (\mu+\varphi)}
          \le \inf_{\varphi\in N} \OPN{\K} \MN{\mu+\varphi}\\ 
          &= \OPN{\K} \inf_{\varphi\in N} \MN{\mu+\varphi}
          =\OPN{\K} \MTN{[\mu]}.
        \end{aligned}
    \end{equation*}
    In addition, since $\K$ is symmetric,
    \begin{equation*}
        \begin{aligned}
          |\innerp{ \mu, \K \nu }| 
          &= \inf_{\varphi\in N} \abs{ \innerp{ \nu, \K(\mu+\varphi) } }
          = \inf_{\varphi\in N}\abs{ \innerp{ \mu+\varphi, \K\nu } } \\
          &\le \inf_{\varphi\in N}\MN{\mu+\varphi} \SN{\K\nu}
          = \MTN{[\mu]} \SN{\K\nu}.
        \end{aligned}
    \end{equation*}
    Therefore, $\norm{ \K \mu }_{\HP}^2
        = \abs{ \innerp{ \mu, \K \mu } }
        \le \OPN{\K} \norm{[\mu]}^2_{\MMM}$.
  \end{proof}

  \begin{lemma}\label{lem:measure_partial_change}
    If $\mu_{\theta}$ is differentiable with respect to $\theta$, then 
    \begin{itemize}
      \item $\partial_i \innerp{ \mu_\theta, x } = \innerp{ \partial_i \mu_\theta, x }$ for $x \in \CC$.
      \item $\partial_i \innerp{ \mu_{\theta}, \K \mu_{\theta} } = 2 \innerp{ \partial_i\mu_{\theta}, \K \mu_{\theta} }$.
      \item Moreover, if $\mu_{\theta}$ is twice differentiable, then
        \begin{equation*}
            \partial_{ij} \innerp{ \mu_{\theta}, \K \mu_{\theta} } = 2 \innerp{ \partial_{ij} \mu_{\theta}, \K \mu_{\theta} } + 2 \innerp{ \partial_i\mu_{\theta}, \K (\partial_j\mu_{\theta}) }.
        \end{equation*}
        \end{itemize}
  \end{lemma}

  \begin{proof}
    Since $\MM=\CC^*$,  
    \begin{equation*}
      \partial_i \innerp{ \mu_\theta, x }
      = \lim_{\delta\to 0} \frac{ \innerp{ \mu_{\theta+\delta e_i}, x } - \innerp{ \mu_\theta, x } }{\delta}
      = \lim_{\delta\to 0} \innerp*{ \frac{\mu_{\theta+\delta e_i}-\mu_\theta}{\delta} , x }
      =\innerp{ \partial_i \mu_\theta, x }.
    \end{equation*}
    Similarly, we obtain 
    \begin{align*}
      \partial_i  \innerp{ \mu_\theta, \K \nu_\theta }
      &= \lim_{\delta\to 0} 
        \frac{ \innerp{ \mu_{\theta+\delta e_i}, \K\nu_{\theta+\delta e_i} } - \innerp{ \mu_\theta, \K\nu_\theta } }{\delta} \\
      &= \lim_{\delta\to 0} 
        \qty{ \innerp*{\frac{\mu_{\theta+\delta e_i} - \mu_\theta}{\delta}, \K \nu_\theta} 
        + \innerp*{\mu_{\theta+\delta e_i}, \K \frac{\nu_{\theta+\delta e_i}-\nu_\theta}{\delta}}}\\
      &= \innerp{\mu_\theta, \K(\partial_i \nu_\theta)} + \innerp{\partial_i \mu_\theta, \K \nu_\theta}.
    \end{align*}
    Moreover, it follows that
    \begin{align*}
      \partial_{ij} \innerp{ \mu_{\theta}, \K \mu_{\theta} }
      &= 2 \, \partial_j \innerp{ \partial_i\mu_{\theta}, \K \mu_{\theta} } \\
      &= 2 \lim_{\delta\to 0} 
        \frac{\innerp{ \partial_i\mu_{\theta+ \delta e_j}, \K\mu_{\theta+ \delta e_j} } 
        - \innerp{ \partial_i\mu_{\theta+ \delta e_j}, \K \mu_{\theta} } }{\delta} \\ 
      &\quad 
        + 2 \lim_{\delta\to 0} 
        \frac{ \innerp{ \partial_i\mu_{\theta+ \delta e_j}, \K\mu_\theta } 
        - \innerp{ \partial_i\mu_{\theta}, \K \mu_{\theta} } }{\delta}\\
      &= 2 \innerp*{ \partial_{ij} \mu_{\theta}, \K \mu_{\theta} } 
        + 2 \innerp*{ \partial_i\mu_{\theta}, \K (\partial_j\mu_{\theta}) }.
    \end{align*}
  \end{proof}

  \begin{lemma}\label{lem:convergence of discret.mu}
    For $\theta\in\Theta$ and $\mu_\theta \in \MM$,
    \begin{itemize}
      \item If $\mu_\theta$ is differentiable with respect to $\theta$, then discretization operator $(\cdot)^n$ for $\mu_\theta$ and the differential operator $\partial_i=\pdv{\theta_i}$ are commutable, i.e., $\partial_i \left(\mu^n_\theta \right) = \left(\partial_i \mu_\theta \right)^n$. 
      \item $\sup_{n\in\N}\MN{\mu^n_\theta} \le \MN{\mu_\theta}$.
      \item The sequence $\{\mu^n_\theta\}_{n\in \mathbb{N}}$ converges 
            in the weak$^{*}$ topology to $\mu_\theta$, 
            i.e.,
            $\innerp{\mu^n_\theta,f} \to \innerp{\mu_\theta,f},\quad n\to\infty$ for any $f\in \CC$. 
            Moreover, if $f\in \K(\M)$ and $\sup_{\theta} \MN{\mu_\theta}<\infty$, 
            then this convergence holds uniformly with respect to $\theta\in\Theta$:
            \begin{equation*}
              \sup_{\theta\in\Theta} \abs{ \innerp{\mu^n_\theta,f} - \innerp{\mu_\theta,f} } \to 0,\quad n\to\infty.
            \end{equation*}
      \item $\OPN{\K^n- \K}\to 0,\quad n\to\infty$.
    \end{itemize}
  \end{lemma}

  \begin{proof}
    It follows from the definition of $\partial_j(\mu^n_\theta)$ that
    \begin{align*}
      \partial_j(\mu^n_\theta) 
      &= \lim_{\gamma\to 0} \frac{\sum_{i=1}^n \mu_{\theta+\gamma e_j}(T_i^n) \, \delta_{t_i^n} - \sum_{i=1}^n \mu_{\theta}(T_i^n) \, \delta_{t_i^n}}{\gamma}\\
      &= \lim_{\gamma\to 0} \sum_{i=1}^n \frac{\mu_{\theta+\gamma e_j}-\mu_\theta}{\gamma}(T_i) \, \delta_{t_i}
      = \sum_{i=1}^n \qty( \partial_j\mu_\theta ) \qty(T_i) \, \delta_{t_i}
      = \qty( \partial_j\mu_\theta )^n.
    \end{align*}
    We see that
    \begin{equation*}
      \MN{\mu^n_\theta}
      = \sup_{\SN{f}\le 1} \abs{ \innerp{ \mu^n_\theta, f } }
      \le \sup_{\SN{f}\le 1} \sum_{i=1}^n \int_{T_i} \abs{ 
      f(t_i^n) } \, \abs{\mu_\theta}(\dd{s})
      \le \MN{\mu_\theta}.
    \end{equation*}
    
    For any $f \in \CC$,
    \begin{equation*}
        \begin{aligned}
          \abs{ \innerp{ \mu^n_\theta-\mu_\theta, f } }
          &= \abs{ \sum_{i=1}^n \mu_\theta(T^n) f(t_i) - \int_T f(t) \, \mu_\theta(\dd{t}) } \\
          &= \abs{ \int_T \qty(f^n(t) - f(t)) \, \mu_\theta(\dd{t}) }
          = \abs{ \innerp{\mu_\theta,f^n-f} }.
        \end{aligned}  
    \end{equation*}
    Since $f$ is uniformly continuous on $T$ 
    and $\displaystyle \lim_{n\to \infty} \sup_{i=1,\dots,n} |t_i-t_{i-1}| =0$, 
    \begin{equation*}
      \innerp{\mu^n_\theta,f} \to \innerp{\mu_\theta,f},\quad n\to\infty.
    \end{equation*}
    Moreover, if there exists $\nu\in\M$ such that $f=\K \nu$, then
    \begin{equation*}
      \sup_{\theta\in\Theta} \abs{ \innerp{\mu^n_\theta - \mu_\theta, f} } =\sup_{\theta\in\Theta} \abs{ \innerp{\mu_\theta, (\K\nu)^n-\K\nu)} } 
      \le \sup_{\theta\in\Theta} \MN{\mu_\theta}\SN{(\K\nu)^n - \K\nu}.
    \end{equation*}
    Therefore, we obtain uniform convergence in $\theta$.

    Since $K(s,t)$ is uniformly continuous and $\displaystyle \lim_{n\to \infty} \sup_{i=1,\dots,n} |t_i-t_{i-1}| =0$,
    \begin{equation*}
        \begin{aligned}
          \OPN{\K^{n}-\K} 
          &= \sup_{\MN{\mu}\le 1} \SN{\left(\K^{n}-\K \right)\mu}
          = \sup_{\MN{\nu}\le 1} \sup_{\MN{\mu}\le 1} \abs{ \innerp{\nu,( \K^{n}-\K ) \mu} }\\
          &\le \sup_{\MN{\nu}\le 1} \sup_{\MN{\mu}\le 1} \sum_{i,j=1}^n \int_{\substack{{s\in T_i}\\ {t\in T_j}}} \abs{ K(t_i, t_j) - K(s,t) } \, \abs{\mu}(\dd{s}) \, \abs{\mu}(\dd{t})\\
          &\le \sup_{\MN{\nu}\le 1} \sup_{\MN{\mu}\le 1} \MN{\mu} \MN{\nu} \max_{i,j=1,\dots,n} \sup_{\substack{{s\in T_i}\\ {t\in T_j}}} \left| K(t_i, t_j) - K(s,t) \right|
          \to 0,
        \end{aligned}
    \end{equation*}
    as $ n\to \infty$.
  \end{proof}

\section{Fundamental Tools}

    In this section, we introduce two fundamental tools.
    
  \begin{proposition}\label{prop:M-estimator}
    Let $ (\Phi_\varepsilon)_{\varepsilon>0}$ be a family of random functions on $\Theta$, and let $\Phi$ be a deterministic function on $\Theta$ such that for every $\delta>0$ and some $\theta_0\in\Theta$;
    \begin{equation}
        \begin{aligned}
            \sup_{\theta\in\Theta} |\Phi_\varepsilon(\theta)-\Phi(\theta)| 
            &\asconv 0, \quad \varepsilon\to 0,\\
            \sup_{\|\theta - \theta_0\|>\delta} \Phi(\theta)
            &< \Phi(\theta_0).
        \end{aligned}
        \label{eq:assump2}
    \end{equation}
    Then, any sequence of estimators $\hat{\theta}_\varepsilon$ satisfying that  $\Phi_\varepsilon ( \hat{ \theta }_\varepsilon ) \ge \Phi_\varepsilon (\theta_0)$ converges almost surely to $\theta_0$.
  \end{proposition}
  \begin{proof}
    The proof is almost the same as for Theorem 5.7 in \textcite{vaart1998asymptotic}. 
    \begin{equation}\label{eq:osaekomi}
      \Phi(\theta_0) - \Phi(\hat{\theta}_\varepsilon )
      \le \Phi_\varepsilon(\hat{\theta}_\varepsilon) - \Phi(\hat{\theta}_\varepsilon) + o(1)\asconv 0,
    \end{equation}
    as $\varepsilon \to 0$, where $o(1)$ means a sequence in $\R$ which converges to $0$.
    By \Cref{eq:assump2}, for any $\delta>0$, there exists $\eta>0$ such that 
    if $\abs{\theta-\theta_0} \ge \delta$, then $\Phi(\theta) < \Phi(\theta_0) -\eta. $        
    By contraposition, if
    $\Phi(\theta_0) - \Phi(\theta) \le \eta$ ,then
    $\abs{\theta-\theta_0}
    < \delta.$
    Then, by \Cref{eq:osaekomi},
    there exists $\varepsilon_0>0$ such that if $ 0<\varepsilon<\varepsilon_0 $, then
    $\Phi(\theta_0) - \Phi(\hat{\theta}_\varepsilon) \le \eta $
    which implies $\hat{\theta}_\varepsilon \to \theta_0$ as $\varepsilon \to 0$.
  \end{proof}
  
  \begin{remark}\label{rmk:DefOfEstimator}
    In order to present a more rigorous proof for \Cref{prop:M-estimator}, 
    we should set $\hat{\theta}_\varepsilon$ in the statement as 
    \begin{equation*}
      \hat{\theta}_\varepsilon \coloneqq 
      \Set{ \theta\in\Theta | \Phi_\varepsilon(\theta) > \Phi_\varepsilon(\theta_0) + o_\varepsilon },
    \end{equation*}
    where $o_\varepsilon$ is a sequence in $\R$ with $o_\varepsilon=o(1)$.
    Then, by the same argument, we obtain 
    \begin{equation*}
      \sup_{\theta\in\hat{\theta}_\varepsilon} \abs{ \theta - \theta_0 }
      \asconv 0
      \quad \text{as} \  
      \varepsilon\to0.
    \end{equation*}
  \end{remark}
  
  \begin{proposition}[Morrey's inequality]\label{prop:Morrey}
    Let $\Theta\subseteq\mathbb{R}^p$ be a closed bounded convex set with smooth boundary and with an interior point.
    For $\mu_\bullet\in\CkThM[1]$, $f\in\CC$ and $q>p$,
    \begin{equation*}
      \sup_{\theta\in\Theta} \abs{ \innerp*{ \mu_\theta, f } }
      \leq C \qty{ \qty( \int_\Theta \abs{ \innerp*{ \mu_\theta, f } }^q  \dd{\theta} )^{1/q}
      + \sum_{i=1}^p \qty( \int_\Theta \abs{ \innerp*{ \partial_i \mu_\theta, f } }^q  \dd{\theta} )^{1/q} },
    \end{equation*}
    where the constant $C$ depends only on $p$, $q$ and $\Theta$.
  \end{proposition}
  
  \begin{proof}
    Let $f\in\CC$.
    \Cref{lem:measure_partial_change} implies that $\innerp*{ \mu_\bullet, f }:\Theta\to\mathbb{R}$, $\theta\mapsto\innerp*{ \mu_\theta, f }$ is differentiable with bounded derivatives, 
    so that $\innerp{\mu_\bullet,f}\in W^{1,q}(\Theta)$ for any $q\geq1$, in particular, for any $q>p$.
    Thus, the conclusion holds by the usual Morrey inequality (see, e.g., Theorem 5.6.5 in \textcite{evans2010partial}).
  \end{proof}

\section{Usual argument to see \Cref{thm:QGAIC}}\label{sec:AnotherProofModelSelection}

  \begin{lemma}\label{lem:r1}
    For $A'_{n,\varepsilon}$ in the proof in \Cref{thm:dis_parametric asymptotic normality}, let us denote by 
    \begin{equation*}
      r^1_{n,\varepsilon} \coloneqq \left(\grad_\theta \Phi_{n,\varepsilon}(\hat\theta_{n,\varepsilon})\right)^\transp (\hat\theta_{n,\varepsilon} -\theta_0)  \idv_{(A'_{n,\varepsilon})^\complement}.
    \end{equation*}
    Under \Cref{assump:Theta,assump:cont,assump:isolated,assump:convex_Theta,assump:Sigma_regular,assump:convergent_rate,assump:MuIsC^2}, it holds that
    \begin{equation*}
      \E \qty[ \abs{ r^1_{n,\varepsilon} } ] = o(\varepsilon^2),
      \quad \varepsilon\to 0, \  n\to\infty.
    \end{equation*}
  \end{lemma}
  
  \begin{proof}
    It follows from \Cref{thm:dis_parametric_consistency} that, for any $\eta>0$,
    \begin{equation*}
      \Prob(\varepsilon^{-2}r_{n,\varepsilon}^1 > \eta) < \Prob((A_{n,\varepsilon})^\complement)\to 0, \quad \varepsilon\to 0, \  n\to \infty, 
    \end{equation*}
    and that, from \Cref{assump:MuIsC^2,lem:convergence of discret.mu},  
    \begin{align*}
      \abs{ \partial_i \Phi_{n,\varepsilon}(\hat\theta_{n,\varepsilon}) }^2 
      &= \abs{ \innerp{\partial_i \hat{\mu}_{n,\varepsilon}^n, h_0+\varepsilon Z} - \innerp{\partial_i \hat{\mu}_{n,\varepsilon}^n, \K \hat{\mu}_{n,\varepsilon}^n} }^2 \\
      &= \abs{ \innerp{\partial_i \hat{\mu}_{n,\varepsilon}^n, \K \mu_{\theta_0}} + \varepsilon \innerp{\partial_i \hat\mu_{n,\varepsilon}^n, Z} - \innerp{\partial_i \hat{\mu}_{n,\varepsilon}^n, \K \hat{\mu}_{n,\varepsilon}^n} }^2 \\
      &\le 2 \varepsilon^2 \innerp*{\partial_i \hat{\mu}_{n,\varepsilon}^n, Z}^2 
      + 2 \innerp{(\partial_i  \hat{\mu}_{n,\varepsilon})^n - \partial_i \mu_{\theta_0} , \K (\hat{\mu}_{n,\varepsilon})^n}^2,
    \end{align*}
    for any $i=1,\dots,p$.
    Since $Z$ is a Gaussian vector, we have that 
    \begin{equation*}
      \E\qty[ \innerp{\partial_i \hat{\mu}_{n,\varepsilon}^n, Z}^2 ] 
      = \innerp{\partial_i \hat{\mu}_{n,\varepsilon}^n, \K \partial_i \hat{\mu}_{n,\varepsilon}^n}
      = \innerp{\partial_i \hat{\mu}_{n,\varepsilon}, \K^n \partial_i \hat{\mu}_{n,\varepsilon}}.
    \end{equation*}
    Therefore 
    $\E[ | \grad_\theta \Phi_{n,\varepsilon}(\hat\theta_{n,\varepsilon}) |^2 ]$ is bounded by \Cref{assump:MuIsC^2,lem:convergence of discret.mu},.
    Thus,  we obtain from \Cref{prop:MomentConvergence} and H\"{o}lder's inequality that
    \begin{align*}
      \E \qty[ \abs{ r_{n,\varepsilon}^1 } ] 
      &= 
        \E\qty[ \qty( \grad_\theta \Phi_{n,\varepsilon}(\hat\theta_{n,\varepsilon}) )^2 ]^{1/2} 
        \E\qty[\abs{\varepsilon^{-1}(\hat\theta_{n,\varepsilon}-\theta_0)}^3 ]^{1/3} 
        \E\qty[ \idv_{(A'_{n,\varepsilon})^\complement} ]^{1/6} \\
      &= 
        \E\qty[ \qty(\grad_\theta \Phi_{n,\varepsilon}(\hat\theta_{n,\varepsilon}) )^2 ]^{1/2} 
        \E \qty[ \abs{ \varepsilon^{-1}(\hat\theta_{n,\varepsilon} -\theta_0) }^3 ]^{1/3} 
        \Prob((A'_{n,\varepsilon})^\complement)^{1/6} \\
      &=o(\varepsilon^2),
    \end{align*}
    as $\varepsilon\to 0$ and $n\to\infty$.
  \end{proof}
  
  \begin{lemma}\label{lem:r2} 
    For any $u\in(0,1)$ and $\widetilde{\theta}_{n,\varepsilon}'\coloneqq u \theta_0 + (1-u) \hat{\theta}_{n,\varepsilon}$, let 
    \begin{align*}
      r_{n,\varepsilon}^2 
      &\coloneqq 
      \frac12 (\hat{\theta}_{n,\varepsilon} - \theta_0)^\transp \left\{\nabla_\theta^2 \Phi(\widetilde{\theta}_{n,\varepsilon}) - \nabla_\theta^2 \Phi(\theta_0)\right\} (\hat{\theta}_{n,\varepsilon}-\theta_0),\\
      r_{n,\varepsilon}^3 
      &\coloneqq 
      \frac12 (\hat{\theta}_{n,\varepsilon} - \theta_0)^\transp \left\{\nabla_\theta^2 \Phi_{n,\varepsilon}(\widetilde{\theta}_{n,\varepsilon}) - \nabla_\theta^2 \Phi(\widetilde{\theta}_{n,\varepsilon})\right\} (\hat{\theta}_{n,\varepsilon}-\theta_0).
    \end{align*}
    Under \Cref{assump:Theta,assump:cont,assump:isolated,assump:convex_Theta,assump:Sigma_regular,assump:MuIsC^3,assump:SmoothBoundaryForMorrey,assump:convergent_rate,assump:MuIsC^2}, it follows that
    \begin{equation*}
      \E\qty[ \abs{ r_{n,\varepsilon}^2 } ] 
      = o(\varepsilon^2) 
      \quad \text{and} \quad
      \E \qty[ \abs{ r_{n,\varepsilon}^3 } ] 
      = o(\varepsilon^2),\quad \varepsilon\to 0, \  n\to\infty.
    \end{equation*}
  \end{lemma}
  
  \begin{proof}
    The function $\theta\mapsto \nabla_{\theta}^2 \Phi(\theta)$ is Lipschitz continuous: there exists $C>0$ such that
    \begin{equation*}
      \abs{ \nabla_{\theta}^2 \Phi(\theta_1) - \nabla_{\theta}^2 \Phi(\theta_2) } 
      < C \abs{ \theta_1 - \theta_2 }, 
    \end{equation*}
    for any $\theta_1, \theta_2 \in \Theta$.
    Therefore, we obtain from \Cref{prop:MomentConvergence} that 
    \begin{equation*}
      \E \qty[ \abs{ r_{n,\varepsilon}^2 } ] 
      \le \frac{C\varepsilon^3}2 
        \E \qty[ \abs{ \varepsilon^{-1}(\hat\theta_{n,\varepsilon} -\theta_0) }^3 ] 
      = \order{\varepsilon^3}
      \quad 
      \text{as} \  \varepsilon\to 0, \  n\to\infty.
    \end{equation*}    

    Moreover, it follows from \Cref{prop:MomentConvergence} and Schwarz's inequality that
    \begin{equation*}
      \E \qty[ \abs{ r_{n,\varepsilon}^3 } ] 
      \le \frac{\varepsilon^2}{2} \E \qty[ \sup_{\theta\in\Theta} \abs{ \nabla_\theta^2 \Phi_{n,\varepsilon}(\theta) - \nabla_\theta^2 \Phi(\theta) }^2 ]^{\frac{1}{2}} \E \qty[ \abs{ \varepsilon^{-1}(\hat\theta_{n,\varepsilon} -\theta_0) }^4 ]^{\frac{1}{2}}.
    \end{equation*}
    Therefore, it suffices to show by \Cref{prop:MomentConvergence} that
    \begin{equation*}
      \E\left[\sup_{\theta\in\Theta} \left| \nabla_\theta^2 \Phi_{n,\varepsilon}(\theta) - \nabla_\theta^2 \Phi(\theta) \right|^2 \right] < \infty,\quad \varepsilon\to 0, \  n\to\infty.
    \end{equation*}

    For $j,k=1,\dots,p$, 
    it is analogous to the proof of \Cref{lem:discrete_twice_nabla_uni_conv} that
    \begin{align*}
      &\sup_{\theta\in\Theta} \left|\partial_{jk} \left\{\Phi_{n,\varepsilon} (\theta)- \Phi(\theta) \right\}\right|^2 \\
      \le& \sup_{\theta\in\Theta}
      \Bigl\{ \varepsilon\innerp{(\partial_{jk} \mu_\theta)^n,Z} + \innerp{(\partial_{jk}\mu_\theta)^n -\partial_{jk}\mu_\theta, \K\mu_0}\\
      &\qquad\qquad\left.
      +\OPN{\K^n-\K}\left(\MN{\partial_{jk}\mu_\theta}\MN{\mu_\theta} + \MN{\partial_{j}\mu_\theta}\MN{\partial_k\mu_\theta}\right)\right\}^2\\
      \le& \sup_{\theta\in\Theta}
      4 \Bigl\{ \varepsilon\innerp{(\partial_{jk} \mu_\theta)^n,Z}^2 + \innerp{(\partial_{jk}\mu_\theta)^n -\partial_{jk}\mu_\theta, \K\mu_0}^2\\
      &\qquad\qquad\left.
      +\OPN{\K^n-\K}^2\left(\MN{\partial_{jk}\mu_\theta}\MN{\mu_\theta} + \MN{\partial_{j}\mu_\theta}\MN{\partial_k\mu_\theta}\right)^2\right\}.
    \end{align*}
    By \Cref{assump:MuIsC^3,assump:Theta,lem:convergence of discret.mu}, the proof ends if we show that there exists $C>0$ independent of $n,j,k$ such that
    \begin{equation}
        \E \qty[ \sup_{\theta\in\Theta} \innerp{(\partial_{jk} \mu_\theta)^n,Z}^2 ] < C
        \label{eq:AICD1Bdd}.
    \end{equation}
    Applying \Cref{prop:Morrey} from \Cref{assump:MuIsC^3}, we see that there exists a constant $C_1>0$ depending only on $p,q,\Theta$ such that
    \begin{equation}
        \begin{aligned}
            & \E \qty[ \sup_{\theta\in\Theta} \innerp{(\partial_{jk} \mu_\theta)^n,Z}^2 ] \\
            &\quad \le 2C_1 \qty{ \E \qty[ \qty( \int_\Theta \abs{ \innerp*{ \partial_{jk}\mu_\theta, Z } }^q  \dd{\theta} )^{2/q} ]
            + \sum_{i=1}^p \E \qty[ \qty( \int_\Theta \abs{ \innerp*{ \partial_{ijk} \mu_\theta, Z } }^q  \dd{\theta} )^{2/q} ] }.
        \end{aligned}
        \label{eq:AICD1BddMor}
    \end{equation}
    Since $\innerp*{ \partial_{ijk} \mu_\theta, Z }\sim N\left(0,\innerp*{\partial_{ijk} \mu_\theta, \K (\partial_{ijk} \mu_\theta)}  \right)$, it follows from \Cref{assump:MuIsC^3} that 
    \begin{equation*}
        \E\left[\innerp*{ \partial_{ijk} \mu_\theta, Z } ^{2q}\right] 
        = 
        \frac{(2q)!}{2^q\cdot q!}
        \innerp*{ \partial_{ijk} \mu_\theta, \K \partial_{ijk} \mu_\theta} ^q
        \le \frac{(2q)!}{2^q\cdot q!} 
        \MN{\partial_{ijk} \mu_\theta}^2 \OPN{\K}, 
    \end{equation*}
    and by \Cref{assump:MuIsC^3,assump:Theta},
    \begin{equation}
        \E\left[\qty( \int_\Theta \abs{ \innerp*{ \partial_{ijk} \mu_\theta, Z} }^q  \dd{\theta} )^{2/q} \right] 
        \le C_1\qty( \int_\Theta \E\left[\innerp*{ \partial_{ijk} \mu_\theta, Z } ^{2q}\right] \dd{\theta} )^\frac{1}{q}
        \leq C_2
        \MN{\partial_{ijk} \mu_\theta}^{2} \OPN{\K},
        \label{eq:AICD1Bdd1}
    \end{equation}
    where the constant $C_2$ depends only on $p,q$ and $\Theta$. Similarly, we obtain
    \begin{equation}
        \E\left[ \qty( \int_\Theta \abs{ \innerp*{ \partial_{jk}\mu_\theta, Z} }^q  \dd{\theta} )^{2/q}\right]
        \leq C_2
        \MN{\partial_{jk} \mu_\theta}^{2} \OPN{\K}.
        \label{eq:AICD1Bdd2}
    \end{equation}
    By combining \cref{eq:AICD1BddMor,eq:AICD1Bdd1,eq:AICD1Bdd2}, we obtain \Cref{eq:AICD1Bdd}. This completes the proof. 
  \end{proof}
  
  \begin{proof}[Proof of \Cref{thm:QGAIC}]
      \begin{align*}
            \Phi_{n,\varepsilon}(\hat{\theta}_{n,\varepsilon})-\Phi(\hat{\theta}_{n,\varepsilon}) 
            &= \qty\Big(\Phi_{n,\varepsilon}(\hat{\theta}_{n,\varepsilon}) - \Phi_{n,\varepsilon}(\theta_0))
                + \qty\Big(\Phi_{n,\varepsilon}(\theta_0) - \Phi(\theta_0))
                + \qty\Big(\Phi(\theta_0) - \Phi(\hat{\theta}_{n,\varepsilon})) \\
            &\eqqcolon D_1 + D_2 + D_3.
      \end{align*}
      By Taylor's formula, there exist  $u_1\in(0,1)$ and $\widetilde{\theta}_{n,\varepsilon}'\coloneqq u_1 \theta_0 + (1-u_1) \hat{\theta}_{n,\varepsilon}$ such that 
      \begin{align*}
          D_1
          =& \left(\grad_\theta \Phi_{n,\varepsilon}(\hat{\theta}_{n,\varepsilon}) \right)^\transp (\hat\theta_{n,\varepsilon} -\theta_0) (\idv_{A'_{n,\varepsilon}} + \idv_{(A'_{n,\varepsilon})^\complement}) 
          + \frac{1}{2} (\hat{\theta}_{n,\varepsilon} -\theta_0)^\transp \left(\nabla_\theta^2 \Phi_{n,\varepsilon}(\widetilde{\theta}_{n,\varepsilon})\right) (\hat{\theta}_{n,\varepsilon} - \theta_0)\\
          =& -\frac{1}{2} (\hat{\theta}_{n,\varepsilon} - \theta_0)^\transp \Sigma(\hat{\theta}_{n,\varepsilon} - \theta_0) + r_{n,\varepsilon}^1 + r_{n,\varepsilon}^2 + r_{n,\varepsilon}^3.
      \end{align*}
      Thus, we see by \Cref{lem:r1,lem:r2} that
      \begin{equation}
          \begin{aligned}
              \varepsilon^{-2}\E\left[D_1 \right] 
              &=-\frac{1}{2}\E\left[\frac{1}{\varepsilon}(\hat{\theta}_{n,\varepsilon} - \theta_0)^\transp \;\Sigma\; \frac{1}{\varepsilon}(\hat{\theta}_{n,\varepsilon} - \theta_0) \right] +o(1)\\
              &= -\frac{1}{2} \tr\left(\Sigma\ \E\left[\frac{1}{\varepsilon}(\hat{\theta}_{n,\varepsilon} - \theta_0)^\transp\;\frac{1}{\varepsilon}(\hat{\theta}_{n,\varepsilon} - \theta_0)\right]\right) +o(1) \\
              &= -\frac{1}{2} \tr\left(\Sigma\Sigma^{-1}\right) + o(1)
              = -\frac{p}{2} + o(1), \quad \varepsilon\to 0, \  n\to\infty.
          \end{aligned}
          \label{eq:AIC_D1}
      \end{equation}
      
      Since $Z$ is a centered Gaussian vector, it follows by \Cref{eq:ConvergenceRateForAIC} and the same argument of \labelcref{eq:NormConvergenceOfMeanByCovSmoothness} that
      \begin{equation} 
          \varepsilon^{-2}\E\left[D_2\right] 
          = \varepsilon^{-2}\innerp{\mu_{\theta_0}, \K (\mu^n_{\theta_0}-\mu_{\theta_0})} 
          - \frac{1}{2\varepsilon^2}\innerp{\mu_{\theta_0}, (\K^n - \K)\mu_{\theta_0}}
          = o(1) 
          \label{eq:AIC_D2} 
      \end{equation}
      as $\varepsilon\to0$, $n\to\infty$.
      
          By \Cref{assump:Theta,assump:MuIsC^2}, $\nabla^2_\theta \Phi_{n,\varepsilon}$ is continuous on the bounded set $\Theta$. Applying Taylor's formula, there exists $u_2\in(0,1)$ and $\theta_{n,\varepsilon}'\coloneqq u_2 \theta_0 + (1-u_2) \hat{\theta}_{n,\varepsilon}$ such that
      \begin{equation*}
          D_3 = \frac{1}{2} (\hat{\theta}_{n,\varepsilon} - \theta_0)^\transp \nabla_\theta^2 \Phi(\theta_0)  (\hat{\theta}_{n,\varepsilon}-\theta_0)
          +\frac{1}{2} (\hat{\theta}_{n,\varepsilon} - \theta_0)^\transp \qty{ \nabla_\theta^2 \Phi(\theta_{n,\varepsilon}') - \nabla_\theta^2 \Phi(\theta_0) } (\hat{\theta}_{n,\varepsilon}-\theta_0).
      \end{equation*}
      Analogously to the discussion to obtain \Cref{eq:AIC_D1},
      \begin{equation*}
          \varepsilon^{-2}\E[D_3] = -\frac{p}{2} + o(1),\quad \varepsilon\to 0, \  n\to \infty.
          \label{eq:AIC_D3}
      \end{equation*}

      By combining \Cref{eq:AIC_D1,eq:AIC_D2}
      \begin{equation*}
          \E \qty[ \Phi_{n,\varepsilon}(\hat{\theta}_{n,\varepsilon}) - \Phi(\hat{\theta}_{n,\varepsilon}) ]
          =\E[D_1] +  \E[D_2] + \E[D_3]
          = -\varepsilon^2 p + o(\varepsilon^2)
      \end{equation*}
      as $\varepsilon\to 0$, $n\to \infty$, and so, we have shown the desired result.
  \end{proof}

  \printbibliography

\end{document}